\documentclass[11pt, reqno, a4paper]{amsart}
\usepackage{amssymb,amsfonts,amsthm,enumerate}
\usepackage[utf8]{inputenc}
\usepackage{listings}
\usepackage{bm}
\numberwithin{equation}{section}
\usepackage[margin=1.2in]{geometry}
\usepackage{bigints}
\usepackage[pagewise]{lineno}

\usepackage{color}
\usepackage[bookmarks]{hyperref}

\allowdisplaybreaks[4]

\newtheoremstyle{myremark}{10pt}{10pt}{}{}{\bfseries}{.}{.5em}{}

 \newtheorem{thm}{Theorem}[section]
 
 \newtheorem{lem}[thm]{Lemma}
 \newtheorem{prop}[thm]{Proposition}
 \theoremstyle{definition}
 \newtheorem{defn}[thm]{Definition}
 \newtheorem{exmp}[thm]{Example}
 \newtheorem{rem}[thm]{Remark}

\allowdisplaybreaks[4]

\title[On Weighted Orlicz-Sobolev inequalities]{On  Weighted Orlicz-Sobolev inequalities}

\author[T. V. Anoop, U. Das, and S. Roy]{T. V. Anoop$^{1,*}$, Ujjal Das$^{2}$, and Subhajit Roy$^{3}$}



{\email{anoop@iitm.ac.in, ujjal.rupam.das@gmail.com, rsubhajit.math@gmail.com}}

\thanks{$^*$Corresponding author.}
\begin{document}

\maketitle

\centerline{$^{1,3}$Department of Mathematics, Indian Institute of Technology Madras,
 }
\centerline{Chennai - 600036, India}
\centerline{$^{2}$Department of Mathematics, Technion - Israel Institute of Technology}
\centerline{Haifa 32000, Israel}






\begin{abstract}
Let $\Omega$ be an open subset of $\mathbb{R}^N$ with $N\geq 2.$ We identify various classes of Young functions $\Phi$ and $\Psi$, and function spaces for a weight function $g$ so that the following weighted Orlicz-Sobolev inequality holds:  
\begin{equation*}\label{ineq:Orlicz}
    \Psi^{-1}\left(\int_{\Omega}|g(x)|\,\Psi(|u(x)| )dx \right)\leq C\Phi^{-1}\left(\int_{\Omega}\Phi(|\nabla u(x)|) dx \right),\;\;\;\forall\,u\in \mathcal{C}^1_c(\Omega),
\end{equation*}
 for some $C>0$.
As an application, we study the existence of eigenvalues for certain nonlinear weighted eigenvalue problems.
\end{abstract}

\section{Introduction}
For an open set $\Omega\subset\mathbb{R}^N$ with $N\geq 2$ and $p,q\in (1,\infty)$, there are several results  available in the literature that provide various weight functions $g\in L^1_\text{loc}(\Omega)$ for which the following weighted Sobolev inequality holds:
\begin{equation}\label{p-qmodular}
     \left(\int_{\Omega}|g(x)| |u(x)|^q dx\right)^\frac{1}{q} \leq C\left(\int_{\Omega}|\nabla u(x)|^p dx \right)^{\frac{1}{p}},\;\;\;\forall\, u\in \mathcal{C}^1_c(\Omega),
\end{equation}
for some $C>0$. For example, see the references listed below for various choices of  $p,\,q,\, N$, and $g$ that ensure \eqref{p-qmodular}:

\begin{itemize}
    \item $p=q=2:$ see \cite{Feffer1983,lam2020, Visciglia2005}.
    \item $p=q:$ see \cite{Chan1985, Edmunds1999}.
  \item  $q\in [p,p^*]:$ Caffarelli-Kohn-Nirenberg \cite{C-K-N1984} proved \eqref{p-qmodular} for $g(x)=|x|^{-\frac{N}{\alpha(p,q)}}$,
where $p^*=\frac{Np}{N-p}$ and $\alpha(p,q)=\frac{Np}{N(p-q)+pq}$. See also \cite{Tarentello2002, Kerman1985, Mazya20002}.
\item   $q\in (0,p^*]:$ authors in \cite{Anoop2021} provide various classes of function spaces for $g$ satisfying \eqref{p-qmodular}, which include most of the weight functions considered in the above references.
\end{itemize} 

The main aim of this manuscript is to generalize 
 \eqref{p-qmodular} by replacing the convex functions $t^p$ and $t^q$ with more general {\it{Young functions}}. A function  $\Phi:[0,\infty)\to [0,\infty]$ is called a Young
function if it admits a representation 
\begin{equation*}
    \Phi(t)=\int_0^t \varphi(s) ds\;\;\;  \text{for}\;\; t\geq 0,
\end{equation*} 
where  $\varphi$ is an increasing right continuous function on $[0,\infty)$ such that $\varphi(t)=0$ if and only if $t=0$. The {\it{complementary}} Young function of $\Phi$ is denoted by $\tilde{\Phi}$ and is defined as 
$$\tilde{\Phi}(t)=\int_0^t\tilde{\varphi}(s) ds,$$ where $\tilde{\varphi}(s)=\sup\{t:\varphi(t)\leq s\}$ is the right continuous inverse of $\varphi$. For example, $A_p(t):=t^p$  with $p\in (1,\infty)$  is a Young function.
A Young function $\Phi$ is said to satisfy the {\it{$\Delta_2$-condition}} ($\Phi\in \Delta_2$) if there exists a constant $C\geq 1$ such that 
\begin{equation*}
    \Phi(2t)\leq C\Phi(t),\;\;\;\forall\,t\geq 0.
\end{equation*}
We say that the Young function $\Phi$ satisfies the {$\Delta^\prime$-\it{condition}} ($\Phi\in \Delta^\prime$) if there exists a constant $C\geq  1$ such that 
\begin{equation}\label{deld1}
    \Phi(st)\leq C\Phi(s)\Phi(t)
\end{equation}
for all $s,t\geq 0$. Notice that if $\Phi\in \Delta^\prime$, then $\Phi\in \Delta_2$. 
Associated to a Young function $\Phi,$ we define  $p^-_\Phi$ and $p^+_\Phi$ (cf. \cite{Lai1993})  as 
\begin{equation}\label{*p-Phi*}
p^-_\Phi:=\inf_{t>0}\frac{t\varphi(t)}{\Phi(t)},\;\;\;\;p^+_\Phi:=\sup_{t>0}\frac{t\varphi(t)}{\Phi(t)}.
\end{equation}
  For any Young function $\Phi$, one can check that  $\Phi(t)\leq t\varphi(t)\leq\Phi(2t)$ for all $t \geq 0$, which implies $p^-_\Phi\geq 1.$  If 
 $\Phi\in\Delta_2$, then $p^-_\Phi,p^+_\Phi\in [1,\infty)$. Also, note that for $\Phi=A_p,$  we have $p^+_\Phi=p^-_\Phi=p.$

In this article, we look for a pair $(\Phi,\Psi)$ of Young functions that satisfy the $\Delta_2$ (or $\Delta^\prime$)-condition and function spaces for weight function $g$ so that the following {\it{weighted Orlicz-Sobolev inequality}} holds for some $C>0$:
\begin{equation}\label{modular}
    \Psi^{-1}\left(\int_{\Omega}|g(x)|\,\Psi(|u(x)|)dx \right)\leq C\Phi^{-1}\left(\int_{\Omega}\Phi(|\nabla u(x)|)dx \right),\;\;\;\forall\,u\in \mathcal{C}^1_c(\Omega).
\end{equation}
\begin{defn}
For a pair of Young functions $(\Phi,\Psi)$, we define the admissible function space for the weight function in \eqref{modular}  as 
\begin{equation*}
    \mathcal{H}_{\Phi,\Psi}(\Omega)=\left\{g\in L^1_\text{loc}(\Omega): g \text{ satisfies} \,\eqref{modular}\right\}.
\end{equation*}
\end{defn}
First, we consider the case $\Phi=\Psi$ of \eqref{modular}. In this case, we find sufficient conditions on  $\Phi$ and admissible function spaces for  $g$ so that the following variant of \eqref{modular} holds:
\begin{equation}\label{modulat phi}
    \int_{\Omega}|g(x)|\,\Phi(|u(x)|)dx \leq C\int_{\Omega}\Phi(|\nabla u(x)|) dx,\;\;\;\forall\,u\in \mathcal{C}^1_c(\Omega).
\end{equation}
 If  $\Omega$ is bounded in one direction, then for any Young function $\Phi$  satisfying the $\Delta_2$-condition, we have the following {\it{Poincar\'e inequality}} (see \cite[Lemma 2.9]{salort2021}, \cite[Section 2.4]{Gossez1974}): 
\begin{equation*}
    \int_{\Omega}\Phi(|u(x)|)dx \leq C\int_{\Omega}\Phi(|\nabla u(x)|) dx,\;\;\; \forall\,u\in \mathcal{C}^1_c(\Omega),
\end{equation*}
where $C$ is a positive constant. From the above inequality, it is clear that \eqref{modulat phi} holds for $g\in L^\infty{(\Omega)}$, i.e., $L^\infty(\Omega)\subset \mathcal{H}_{\Phi,\Phi}(\Omega)$. 
In \cite{kal2008}, \eqref{modulat phi} was proved for $N=1$ for certain $\Phi$ and $g$. For $N\geq 3$ and $u\in \mathcal{C}_c(\mathbb{R}^N)$, consider 
the Riesz potential operator given by $$P(u)(x)=\int_{\mathbb{R}^N}\frac{|u(y)|}{|x-y|^{N-1}}dy.$$
Since $|u(x)|\leq \frac{1}{N\omega_N}P(|\nabla u(x)|)$ (see \cite[equation 1.5]{Anoop2021}),  \eqref{modulat phi} easily follows from the following convolution inequality: 
 \begin{equation}\label{fund2}
   \int_{\mathbb{R}^N}|g(x)|\,\Phi\left(|P(u)(x)|\right)dx \leq C\int_{\mathbb{R}^N}\Phi( |u(x)|) dx, \;\;\;\forall\,u\in \mathcal{C}_c(\mathbb{R}^N),
\end{equation}
where $C$ is a positive constant. Thus, for a given $\Phi$, if $g$ satisfies \eqref{fund2}, then $g$ satisfies \eqref{modulat phi}, i.e. $g\in \mathcal{H}_{\Phi,\Phi}(\Omega)$.  Many authors provided various sufficient conditions on $g$ and $\Phi$ so that \eqref{fund2} holds. For example, see \cite{Sawyer1992,Sawer1988} for $\Phi=A_p$, and  \cite[Theorem 2]{Lai1993} for more general Young function $\Phi$ with $\Phi,\tilde{\Phi}\in \Delta_2$.

 To state our first result, we make the following assumption on $\Phi$:
 \begin{equation} \tag{H1} \label{H1}
\int_{0}^{1}\left(\frac{s}{\Phi(s)}\right)^{\frac{1}{N-1}} ds < \infty.\end{equation}
One can show that a Young function $\Phi$ with $p^+_\Phi<N$ satisfies \eqref{H1} (see \eqref{eqn-delta2-2}).   Also, $\Phi=A_p$ satisfies \eqref{H1} if and only if $p<N$. Indeed, the condition \eqref{H1} plays the role of ``the dimension restriction in Sobolev inequalities" for general $\Phi$, for instance, see (\cite[Theorem 1]{cianchi1996}, \cite[Theorem 1]{Cianchi2000}). For a Young function $\Phi$, we consider the following  Young function (cf. \cite{cianchi1996}):
   \begin{equation}\label{Phi--N}
     \Phi_N(t)= \int_0^ts^{N^\prime-1}\left(H_\Phi^{-1}\left(s^{N^\prime}\right)\right)^{N^\prime}  ds\; \;\; \text{for}\;t\geq 0, 
      \end{equation}
      where $H_\Phi^{-1}$ is the inverse of $H_\Phi(t)=\displaystyle\int_0^t\frac{\tilde{\Phi}(s)}{s^{1+N^\prime}}ds$ and $N^\prime=\frac{N}{N-1}$. Now we define 
    \begin{equation}\label{B-phi}
B_\Phi=\Phi_N\circ\Phi^{-1}.
      \end{equation}
 In general, $B_\Phi$ need not be a Young function. However, we provide a sufficient condition on $\Phi$ so that
  $B_\Phi$ is a Young function (see Lemma \ref{B-Phi-Young}). Let  $L^\Phi(\Omega)$ denotes the Orlicz space generated by $\Phi.$ Then we have the following result:
 \begin{thm}\label{thm:3}
 Let $\Omega$ be an open subset of $\mathbb{R}^N$, and $\Phi$ be a Young function such that  $B_\Phi$ is a Young function and $\Phi,\tilde{\Phi}\in\Delta^\prime.$  In addition, assume that $\Phi$ satisfies \eqref{H1} when $|\Omega|=\infty$. If $g \in L^{\tilde{B}_{\Phi}}(\Omega)$, then there exists  $C=C(N,\Phi)>0$ so that
   \begin{equation*}
 \int_{\Omega}|g(x)|\,\Phi(|u(x)|)dx\leq C\|g\|_{L^{\tilde{B}_\Phi}(\Omega)}\int_{\Omega}\Phi(|\nabla u(x)|)dx,\;\;\; \forall \,u\in \mathcal{C}^1_c(\Omega).
 \end{equation*} 
 \end{thm}
 Our proof of the above theorem is based on the embedding of the {\it{Beppo-Levi space $\mathcal{D}_0^{1,\Phi} (\Omega)$}} (the completion of $\mathcal{C}_c^1(\Omega)$ 
 with respect to the norm $\|u\|:=\|\nabla u\|_{L^\Phi(\Omega)})$ into the Orlicz space $L^{\Phi_N}(\Omega)$, due to Cianchi \cite[Theorem 1]{cianchi1996}.

Next, for a Young function $\Phi$ satisfying certain conditions, we use two different methods to provide {\it{Orlicz-Lorentz}} type admissible spaces for $g$ so that \eqref{modulat phi} holds. The first method uses the optimal embedding obtained by Cianchi \cite[Theorem 1.1]{Cianchi2004}. Whereas the second method is based on the{\it{ Muckenhoupt}} type condition for the one-dimensional weighted Hardy inequalities (see \cite[Theorem 5]{LQ}) 
involving the Young function.

To describe the first method, we introduce a {\it{rearrangement-invariant Banach function space}} 
associated to $\Phi$. Let $\Omega\subset\mathbb{R}^N$ be an open set, and $\mathcal{M}(\Omega)$ be the set of all extended real-valued Lebesgue measurable functions that are finite a.e. in $\Omega$. For a Young function $\Phi$ satisfying $p^+_\Phi<N$ and $g\in \mathcal{M}(\Omega)$, we define $$\|g\|_{L^{\Phi,\infty}(\Omega)}=\sup_{0<s< |\Omega|}\left\{\frac{g^{**}(s)}{\Phi(s^{-\frac{1}{N}})}\right\},$$
    where $g^{**}$ is the maximal function (see Definition \ref{maximal}) of the one-dimensional decreasing rearrangement $g^*$ of $g$. Now, we consider the following function space:
     \begin{equation}\label{Phi-infty}
     L^{\Phi,\infty}(\Omega)=\left\{g\in \mathcal{M}(\Omega) :\|g\|_{L^{\Phi,\infty}(\Omega)}<\infty\right\}.
    \end{equation}
One can verify  that $ L^{\Phi,\infty}(\Omega)$ is a rearrangement-invariant Banach function space with respect to $\|g\|_{L^{\Phi,\infty}(\Omega)}$ (see $(i)$ of Remark \ref{BF-S}). In the following theorem, we obtain the admissibility of the space $L^{\Phi,\infty}(\Omega)$ for the weight function $g$ under certain assumptions on $\Phi$.
 \begin{thm}\label{larentzthm} 
 Let $\Omega$ be an open subset of $\mathbb{R}^N$. Let $\Phi$ be a Young function such that $\Phi\in \Delta_2$, $\tilde{\Phi}\in\Delta^\prime$, and $p^+_\Phi<N$. If $g\in L^{\Phi,\infty}(\Omega)$, then there exists $C=C(N,\Phi)>0$ so that
 \begin{equation*}
\int_{\Omega}|g(x)|\,\Phi(|u(x)|)dx \leq C\|g\|_{L^{\Phi,\infty}(\Omega)}\int_{\Omega}\Phi(|\nabla u(x)|) dx,\;\;\;\forall\,u\in \mathcal{C}^1_c(\Omega).
 \end{equation*}
   \end{thm}
    \begin{rem}
The case $p^+_\Phi\geq N$ is open for general $\Phi$. However, we have some results for this case when $\Phi=A_p$ (see Remark \ref{exmp5.1}).
        \end{rem}
For the second method, we associate two functions $Q_\Phi$ and $\eta_\Phi$ to a Young function $\Phi$ as below:
 \begin{equation}\label{eqA*}
     Q_\Phi(s)=\Phi\left(\zeta(s)\right)\tilde{\Phi}\left(\frac{1}{\Phi(\zeta(s))}\right),\;\;\;s>0,\,\,\text{where}\,\zeta(s)=s^{\frac{1}{N}-1},
 \end{equation} 
 \begin{equation}\label{eta-Phi}
     \eta_\Phi(r)=r\varphi\left( \int_r^{|\Omega|}\frac{1}{Q_\Phi(s)}ds\right),\;\;\;r\in (0,|\Omega|).
 \end{equation}
For  $\Phi$ satisfying  the following condition
\begin{equation*}
    \tag{H2}\label{8} \lim_{r\to 0}\eta_\Phi(r)<\infty,
\end{equation*}
and $g\in \mathcal{M}(\Omega)$, we define   
\begin{equation*}
\|g\|_{X_\Phi(\Omega)}=\sup_{0<r<|\Omega|}\left\{g^{**}(r)\eta_\Phi(r)\right\}.
  \end{equation*}
  Now, we consider the following function space:
\begin{equation}\label{Phi-Phi}
  X_{\Phi}(\Omega)=\left\{g\in \mathcal{M}(\Omega) :\|g\|_{X_{\Phi}(\Omega)}<\infty\right\}.
\end{equation}
 It can be verified that $X_{\Phi}(\Omega)$ is a rearrangement-invariant Banach function space with respect to $\|g\|_{X_\Phi(\Omega)}$ (see $(ii)$ of  Remark \ref{BF-S}). 
\begin{thm}\label{cor}
  Let $\Omega$ be an open subset of $\mathbb{R}^N$. Let $\Phi$ be a Young function satisfying \eqref{8} and $\Phi\in \Delta^\prime$,   $\tilde{\Phi}\in\Delta_2$.
If $g\in X_\Phi(\Omega)$, then there exists $C=C(N,\Phi)>0$ so that
 \begin{equation}\label{result-1.6}
   \int_{\Omega}|g(x)|\,\Phi(|u(x)|)dx\leq C \|g\|_{X_\Phi(\Omega)}\int_{\Omega}\Phi(|\nabla u(x)|)dx,\;\;\; \forall\, u\in \mathcal{C}^1_c(\Omega).
  \end{equation}
\end{thm}
Next, we consider the case when $\Phi$ and $\Psi$ are not necessarily equal. Towards this, first we consider two cases: $(i)$ when $\Phi$ dominates $\Psi$ globally, $(ii)$ when $\Phi$ dominates $\Psi$ near infinity and  $\Omega$ is bounded.
We say {\it{$\Phi$ dominates $\Psi$ globally}}  if
there exists a constant $C>0$ such that for all $t\geq 0,$
\begin{equation*}
    \Psi(t)\leq \Phi(Ct).
 \end{equation*}
 Similarly, we say {\it{$\Phi$ dominates $\Psi$ near infinity}
 } if there exists $t_0>0$ such that the above inequality holds for all $t\geq t_0.$ 
 Now, we state our result.
 \begin{thm}\label{thm-sim}
Let   $\Omega$ be an open subset of $\mathbb{R}^N$, and $g\in L^1_\mathrm{loc}(\Omega)$ satisfies  \eqref{modulat phi} with $\Phi\in \Delta_2$ and  
 $\tilde{\Phi},\Psi\in \Delta^\prime$. In addition, if $\Phi$ and $\Psi$ satisfy one of the following conditions:
     \begin{enumerate}[(i)]
         \item  $\Phi$ dominates $\Psi$ globally,
         \item $\Phi$ dominates $\Psi$ near infinity and $\Omega$ is bounded,
     \end{enumerate} 
then $g$ satisfies \eqref{modular}.
 \end{thm}
In the next theorem, we consider the case when $\Psi$ is not necessarily dominated by $\Phi$. Recall that, if  $|\Omega|<\infty$ and $N<p$, then by Sobolev inequality for $g\in L^1(\Omega)$ and $q\in (1,\infty)$ we have
$$  \left(\int_{\Omega}|g(x)||u(x)|^qdx\right)^\frac{1}{q}\leq C\|g\|^\frac{1}{q}_{L^1(\Omega)}
   \left(\int_{\Omega}|\nabla u(x)|^pdx\right)^\frac{1}{p},\;\;\; \forall \,u\in \mathcal{C}^1_c(\Omega),$$ for some  $C>0$. For an analogous result with more general $\Phi$, we introduce the following condition:
\begin{equation*}
\tag{H3}\label{eqG1}
\int_{1}^{\infty}\left(\frac{s}{\Phi(s)}\right)^{\frac{1}{N-1}} ds < \infty.
  \end{equation*}
A Young function $\Phi$ with $p^-_\Phi>N$ satisfies \eqref{eqG1} (see \eqref{eqn-delta2-2}). Also,  $\Phi=A_p$ satisfies \eqref{eqG1} if and only if $p>N$.  Indeed, the condition \eqref{eqG1} plays the role of ``the dimension restriction ($N<p$) 
in Sobolev inequalities'' for general $\Phi$, see (\cite[Theorem 1a]{Cianchi2-1996}).  
\begin{thm}\label{thm-bdd}
 Let $\Omega$ be an open subset of $\mathbb{R}^N,$ and  $\Phi,\Psi$ be Young functions such that  
 $\Phi,\,\Psi\in\Delta_2$, $\tilde{\Phi} \in\Delta^\prime,$ and $\Phi$ satisfies \eqref{eqG1}. Assume that $\Phi$ satisfies \eqref{H1} when $|\Omega|=\infty$. If $g \in L^1(\Omega)$, then there exists  $C=C(N,\Phi,\Psi)>0$ so that
  \begin{equation*}
\Psi^{-1}\left(\int_{\Omega}|g(x)|\,\Psi(|u(x)|)dx\right)\leq C\max\left\{\|g\|^{1/p^-_\Psi}_{L^1(\Omega)},\|g\|^{1/p^+_\Psi}_{L^1(\Omega)}\right\}
     \Phi^{-1}\left(\int_{\Omega}\Phi(|\nabla u(x)|)dx\right),
  \end{equation*}
for all $u\in \mathcal{C}^1_c(\Omega)$.
\end{thm} 
Our next theorem provides a more general Orlicz-Lorentz type admissible space for $g$ so that \eqref{modular} holds. For $\Phi$ and $\Psi$, we define 
\begin{equation}\label{eta-phi,psi}
    \eta_{\Phi,\Psi}(r)=r\Psi\left( \|\zeta\|_{L^{\tilde{\Phi}}((r,\,|\Omega|))}\right),\;\;\; r\in (0,|\Omega|),
\end{equation}
  where $\zeta(s)=s^{\frac{1}{N}-1}$ and $\|\cdot\|_{L^{\tilde{\Phi}}((r,\,|\Omega|))}$ is the Luxemburg norm (see \eqref{defn-Luxemberg}). Assume that 
  \begin{equation*}
    \tag{H4}\label{7}  \lim_{r\to 0}\eta_{\Phi,\Psi}(r)<\infty.
  \end{equation*}
  Now for $g\in \mathcal{M}(\Omega),$ define 
 \begin{equation*}
      \|g\|_{X_{\Phi,\Psi}(\Omega)}=\sup_{0<r<|\Omega|}\left\{g^{**}(r)\eta_{\Phi,\Psi}(r)\right\},
  \end{equation*}
\begin{equation}\label{Phi-Psi}
       X_{\Phi,\Psi}(\Omega)=\left\{g\in \mathcal{M}(\Omega):\|g\|_{ X_{\Phi,\Psi}(\Omega)}<\infty\right\}.
    \end{equation}
One can verify that $X_{\Phi,\Psi}(\Omega)$ is a rearrangement-invariant Banach function space  with respect to $\|g\|_{ X_{\Phi,\Psi}(\Omega)}$ (see $(iii)$ of Remark \ref{BF-S}). To state our next result, we define the notion of {\it{super-additivity}}. 
\begin{defn}\label{dfn1}
A function $f:[0,\infty)\rightarrow [0,\infty]$ is said to be  super-additive 
if there exists a constant $C>0$ such that $$\sum_{i=1}^{\infty} f(a_i)\leq C \, f\left(\sum_{i=1}^\infty a_i\right),$$ for every summable sequence $(a_i)$ in $[0,\infty).$
\end{defn}
\noindent Note that, any convex function  $f:[0,\infty)\rightarrow [0,\infty]$ 
satisfies the above inequality with $C=1$ if $f(0)=0$.
 In particular, $A_q\circ A_p^{-1}$ is super-additive when $q\geq p$. Now we state our result:
\begin{thm}\label{theorem3} Let $\Omega$ be an open subset of $\mathbb{R}^N.$ Let $\Phi$ and $\Psi$  be Young functions satisfying \eqref{7} and   $\Phi,\tilde{\Phi},\Psi\in\Delta^\prime$, and $\Psi\circ \Phi^{-1}$ be super-additive. If $g \in  X_{\Phi,\Psi}(\Omega)$, then there exists $C=C(N,\Phi,\Psi)>0$ so that
  \begin{equation}\label{eqn-thm3}
\Psi^{-1}\left(\int_{\Omega}|g(x)|\,\Psi(|u(x)|)dx\right)\leq C B \Phi^{-1}\left(\int_{\Omega}\Phi(|\nabla u(x)|)dx\right),\;\;\; \forall\,u\in \mathcal{C}^1_c(\Omega),
  \end{equation}
  where $B=\max\left\{\|g\|^{1/p^-_\Psi}_{X_{\Phi,\Psi}(\Omega)},\|g\|_{X_{\Phi,\Psi}(\Omega)}^{1/p^+_\Psi}\right\}$.
  \end{thm}
Our following result provides a necessary and sufficient condition (analogous to {\it{Ma\`zya's capacity}} condition, see \cite[Theorem 8.5]{Mazya20002}) on $g$ so that \eqref{modular} holds.
\begin{defn}[\textbf{Orlicz-Sobolev capacity}]
Let $\Omega$ be an open subset of $\mathbb{R}^N$ and $K$ be a compact subset of $ \Omega$. Then the {\it{$\Phi$-capacity}} of $K$ with respect to $\Omega$ is defined as $$\text{Cap}_{\Phi}(K,\Omega)=\inf\left\{   \int_{\Omega}\Phi(|\nabla u|)dx: u\in \mathcal{C}_c^1(\Omega),\ u\geq 1\  \ on\; K \ \right\}.$$
We refer to  \cite{salort2021,Mazya1985, Rudd2005} for further information on the $\Phi$-capacity.
\end{defn}
\noindent For $\Phi=A_p$ and $\Psi=A_q$, using $\Phi$-capacity, Ma\`zya has provided a necessary and sufficient condition \cite[Theorem 8.5]{Mazya20002} on $g$ so that \eqref{modular} holds.
 He proved that for $1<p\leq q<\infty$, $g$ satisfies \eqref{modular} if and only if 
 there exists a constant $D>0$ such that, for every compact subset $K$ of $\Omega$, $$\int_K|g(x)|dx\leq D\left(\text{Cap}_{\Phi}(K,\Omega)\right)^\frac{q}{p}.$$
The following result is a generalization in terms of general Young functions.
\begin{thm}\label{capacity}
 Let $\Omega $ be an open subset of $\mathbb{R}^N$ and $g\in L^1_{\mathrm{loc}}(\Omega)$. Let $\Phi$ and $\Psi$ be Young functions such that $\Psi\circ\Phi^{-1}$ is super-additive, $\Psi\in \Delta_2,$ and  $\Phi,\tilde{\Phi},\tilde{\Psi}\in\Delta^\prime$. Then the following two conditions are equivalent:
 \begin{enumerate}[(i)]
     \item there exists a constant $C>0$ such that \eqref{modular}
 holds.
 \item there exists a constant $D>0$ such that, for every compact subset $K$ of $\Omega$, 
 \begin{equation*}
   \int_K|g(x)|dx\leq D\,\Psi\circ \Phi^{-1}\left(\mathrm{Cap}_{\Phi}(K,\Omega)\right).  
 \end{equation*}
 \end{enumerate}
Furthermore, for the best constants $C$ and $D$, there exists a constant $C_1=C_1(\Phi,\Psi)>0$  such that $\Psi(CC_1)\leq D\leq \max\{C^{p^-_\Psi},C^{p^+_\Psi}\}/\Psi(1)$.
\end{thm}
\begin{rem}\label{cap-rem}
  Motivated by \cite{Ujjal2021}, for $g\in L^1_\text{loc}(\Omega)$, one can define $$\|g\|_{\mathcal{H}_{\Phi,\Psi}(\Omega)}=\sup\left\{\frac{\int_{K}|g(x)|dx}{\Psi\circ \Phi^{-1}\left(\text{Cap}_{\Phi}(K,\Omega)\right)}:K \,\text{is a 
compact subset  of} \,\, \Omega\right\}.$$
It is easy to see that ${\mathcal{H}_{\Phi,\Psi}(\Omega)}=\{g\in L^1_\text{loc}(\Omega): \|g\|_{\mathcal{H}_{\Phi,\Psi}(\Omega)}<\infty\}$ and $\|g\|_{\mathcal{H}_{\Phi,\Psi}(\Omega)}$ defines a Banach function norm on ${\mathcal{H}_{\Phi,\Psi}(\Omega)}$. 
\end{rem}

Next, as an application of weighted Orlicz-Sobolev inequalities, we study the following weighted eigenvalue problem:
\begin{equation}\label{eqJ}
  -\Delta_{\varphi}u=\lambda g(x)\psi(|u|)\frac{u}{|u|},\;\;\; u\in \mathcal{D}^{1,\Phi}_0(\Omega)\setminus\{0\},
\end{equation} 
where $g\geq 0,$ 
$\varphi=\Phi^\prime$, $\psi=\Psi^\prime$, and $\Delta_{\varphi}$ is the $\Phi$-Laplacian operator defined as $\Delta_{\varphi}u= \text{div} \left(\varphi(|\nabla u|)\frac{\nabla u}{|\nabla u|}\right)$.
We say that $\lambda$ is an eigenvalue of \eqref{eqJ} if there exists  $u\in \mathcal{D}_0^{1,\Phi}(\Omega)\setminus\{0\}$ so that 
\begin{equation*}\label{eqn1.18}
    \int_{\Omega}\varphi(|\nabla u|)\frac{\nabla u}{|\nabla u|}\cdot\nabla v\, dx=\lambda\int_{\Omega}g\psi(|u|)\frac{uv}{|u|}dx,\;\;\;\forall\, v\in \mathcal{D}^{1,\Phi}_0(\Omega).
\end{equation*}
We call $u$ is an eigenfunction corresponding to $\lambda$. 
For $\Phi=A_p$ and $\Psi=A_q$, \eqref{eqJ} reduce to 
\begin{equation}\label{eqnJ*}
     -\text{div}(|\nabla u|^{p-2}\nabla u)=  \frac{\lambda q}{p} g(x)|u|^{q-2}u.
 \end{equation}
The existence of eigenvalues of \eqref{eqnJ*} was studied in \cite{Anane1987, Anoop2021, Anoop2015} 
and the references therein.

For the existence of the eigenvalues of \eqref{eqJ}, we consider the functionals $J_\Phi,G_\Psi:\mathcal{D}_{0}^{1,\Phi}(\Omega)\rightarrow \mathbb{R}$  defined as
 \begin{equation}\label{eqn-G-Psi}
     J_\Phi(u)=\int_{\Omega}\Phi(|\nabla u|) dx,\;\;\;  G_{\Psi}(u)=\int_{\Omega}g\Psi(|u|)dx,\;\;\; u\in\mathcal{D}_{0}^{1,\Phi}(\Omega).
 \end{equation}
Given $g\geq 0$ and $r>0$, we define
\begin{equation}\label{minimizer-attain}
  \lambda_1(r)=\inf\left\{J_\Phi(u):u\in N_r\right\},\,\,\text{where}\,\, N_r=\left\{u\in \mathcal{D}^{1,\Phi}_0(\Omega):G_\Psi(u)=r\right\}.
\end{equation}
Considering $\Phi,\Psi\in\Delta_2$, and $g$ satisfies \eqref{modular}, one can verify that  $J_\Phi$ and $G_\Psi$ are  Fréchet  derivable (see \cite[Proposition 2.17]{salort2021}, \cite[Lemma A.3]{Fukagai2006}) with derivatives 
given by $$\langle J_\Phi^{\prime}(u), v \rangle=\int_{\Omega}\varphi(|\nabla u|)\frac{\nabla u}{|\nabla u|}\cdot\nabla v\, dx, \;\;\;\;\langle G_\Psi^{\prime}(u), v\rangle=\int_{\Omega}g\psi(|u|)\frac{uv}{|u|}dx.$$
Now, if $\lambda_1(r)$ is attained for some $u\in N_r$ and $G_\Psi^\prime(u)\neq 0$, then by the {\it{Lagrange multipliers theorem}} (see \cite[Theorem 4]{Browder1965}, Lemma \ref{lagrange}), there exists $\lambda=\tilde{\lambda}_1(r)>0$ such that $u$ solves 
  $$\langle J_\Phi^{\prime}(u), v\rangle=\tilde{\lambda}_1(r)\langle G_\Psi^{\prime}(u), v\rangle,\;\;\;\forall\, v\in \mathcal{D}^{1,\Phi}_0(\Omega).$$ 
 Thus $\lambda=\tilde{\lambda}_1(r)$ is an eigenvalue of \eqref{eqJ} with eigenfunction $u\in N_r$. For $\Phi=A_p$ and $\Psi=A_q$, using the homogeneity of $J_\Phi$ and $G_\Psi$, it is easy to verify that 
 \begin{equation*}
     \lambda_1(r)=\frac{qr}{p}\tilde{\lambda}_1(r)=r^\frac{p}{q}\lambda_1(1),\;\;\; \forall\,r>0.
 \end{equation*}
Hence, it is enough to consider \eqref{minimizer-attain} only at the level $r=1$ in this case. However, to deal with general $\Phi$ and $\Psi$, due to the lack of homogeneity of $J_\Phi$ and $G_\Psi$, we require to consider \eqref{minimizer-attain} at each level $r>0$.  

One of the sufficient conditions that ensure the existence of a minimizer for \eqref{minimizer-attain} is the compactness of $G_\Psi$ in $\mathcal{D}_0^{1,\Phi}(\Omega)$ (i.e., $u_n\rightharpoonup u$ weakly in $\mathcal{D}_0^{1,\Phi}(\Omega)$ implies $G_\Psi(u_n)\to G_\Psi(u)$). Many authors proved the compactness of the  map $G_\Psi$ in $\mathcal{D}_0^{1,\Phi}(\Omega)$
under various assumptions on $\Phi$, $\Psi$, and $g$. For bounded $\Omega$ and $\Phi=\Psi$, see \cite[$g\equiv 1$]{Vesa1999}, \cite[$g\in L^\infty(\Omega)$]{Gossez2002}. For $\Omega=\mathbb{R}^N$, $\Phi=\Psi, \,N>p^+_\Phi$, and $g\in L^r(\Omega)\cap L^\infty(\Omega)$ where $r=r(\Phi)>0$, see \cite{Waldo2017,Arriagada2019}.
For bounded $\Omega$ and  $\Psi\prec\prec\Phi_*$ (where
$\Phi_*$ is the {\it{Sobolev conjugate}} of $\Phi$ \cite[Page 248]{Adams1975}), see \cite[$g\equiv 1$]{Garc1999}, and \cite[$g\in L^{\tilde{A}}(\Omega)$, where $A=\Phi_*\circ\Phi^{-1}$]{SA}. Our next result allows $\Omega$ to be a general domain and proves the compactness of $G_\Psi$ for a larger class of $g$. We require the following compatibility condition on Young functions. 
\begin{defn}\cite[Page 231]{Adams1975})
    Let $\Phi$ and $\Psi$ be two Young functions. We say {\it{$\Psi$ increases essentially more slowly than $\Phi$ near infinity $(\Psi\prec\prec \Phi)$}} if for all $k>0$ it holds that
\begin{equation*}
     \lim_{t\to \infty}\frac{\Psi(kt)}{\Phi(t)}=0.
\end{equation*}
 Clearly, $\Phi\prec\prec\Psi$ implies that $\Phi$ dominates $\Psi$ near infinity, but the converse does not hold.
\end{defn}
For our next theorem, we consider  a Banach function space $V$ satisfying the following assumptions:
 \begin{enumerate}
     \item[\textbf{\textit{A-1:}}] $\mathcal{C}_c(\Omega)\subset V$;
     \item[\textbf{\textit{A-2:}}] there exists a function $\delta: [0,\infty) \rightarrow [0,\infty)$ satisfying $\delta(t) \rightarrow 0$ as $t \rightarrow 0$ and the following inequality holds for any $g\in V:$
\begin{equation}\label{eqK}
\Psi^{-1}\left(\int_{\Omega}|g(x)|\,\Psi(|u(x)|) dx\right)\leq \delta\left(\|g\|_{V}\right)\Phi^{-1}\left(\int_{\Omega}\Phi(|\nabla u(x)|) dx\right ),\;\;\; \forall\, u\in \mathcal{C}^1_c(\Omega).
\end{equation}
 \end{enumerate}
\begin{thm}\label{solu} 
Let $\Omega$ be an open subset of $\mathbb{R}^N$, and $\Phi,\Psi$ be Young functions such that $\Phi,\tilde{\Phi},\Psi\in\Delta_2$. Further assume that  $\Phi$ satisfies \eqref{H1} when $|\Omega|=\infty.$ Let one of the following conditions hold:
\begin{enumerate}[(i)]
\item $\Psi\prec\prec \Phi_N$,
\item $\Phi$ satisfies \eqref{eqG1}.
\end{enumerate}
In addition, we assume that $V$ satisfies the above assumptions   $\textbf{A-1}$  and $\textbf{A-2}$. 
Then the map $G_{\Psi}$ is compact in $\mathcal{D}^{1,\Phi}_0(\Omega)$ for any $g\in \mathcal{F}_V(\Omega):=\overline{\mathcal{C}_c(\Omega)}^{V}$.
Moreover, if  $g\in \mathcal{F}_V(\Omega)$ is non-zero non-negative, then for each $r>0$, there exists an eigenvalue $\lambda=\tilde{\lambda}_1(r)>0$ of  \eqref{eqJ} such that the corresponding eigenfunction is non-negative.
\end{thm}
\begin{rem} We denote the function space for the weights considered in  Theorem \ref{thm:3}-Theorem \ref{theorem3} by $V.$
\begin{enumerate}[(i)]
     \item  
     Since $\Phi\prec\prec\Phi_N$ (see \cite[Page 133]{Rudd2005}), we can take $\Psi=\Phi$ in the above theorem.  
 Notice that,
 \begin{itemize}
     \item  Theorem \ref{thm:3}: $V=L^{\tilde{B_\Phi}}(\Omega)$ and $\delta(t)\asymp t,$
     \item   Theorem \ref{larentzthm}: $V=L^{\Phi,\infty}(\Omega)$ and $\delta(t)\asymp t,$ 
     \item  Theorem \ref{cor}: $V=X_\Phi(\Omega)$ and $\delta(t)\asymp t.$
 \end{itemize}
Thus for $V$ as considered above, Theorem \ref{solu} together with Theorem \ref{thm:3} or  Theorem \ref{larentzthm} or  Theorem \ref{cor} ensures that the map $G_\Phi$ is compact in  $\mathcal{D}^{1,\Phi}_0(\Omega)$ for any $g\in\mathcal{F}_V(\Omega)$. 
    \item 
In Theorem \ref{thm-bdd}, we have $V=L^1(\Omega)$ and $\delta(t)\asymp\max\{t^{1/p^-_\Phi},t^{1/p^+_\Phi}\}$. In Theorem \ref{theorem3}, we have $V=X_{\Phi,\Psi}(\Omega)$ and $\delta(t)\asymp\max\{t^{1/p^-_\Phi},t^{1/p^+_\Phi}\}.$
\item If $\mathcal{C}_c(\Omega)$ is dense in $V$, then $\mathcal{F}_V(\Omega)$ coincides with $V$.  For $V=L^{\tilde{B_\Phi}}(\Omega)$, we have $\mathcal{F}_V(\Omega)=L^{\tilde{B_\Phi}}(\Omega).$
Similarly, if $V=L^1(\Omega)$, then  $\mathcal{F}_V(\Omega)=L^1(\Omega)$.
\end{enumerate}

 \end{rem}

The rest of this article is organized in the following way: In section 2, we recall some properties of the Young function, symmetrization, and the notion of Banach function space and collect some known results. Section 3, Section 4, and Section 5 contain the proof of Theorem \ref{thm:3}-Theorem \ref{solu}. Section 6 includes some examples and concluding remarks. 
\section{preliminary}
We enlist some of the notations and conventions used in this article:
\begin{itemize}
    \item  $\mathcal{C}_c^1(\Omega)$ is the set of continuously differentiable functions with compact support.
    \item For $p\in (1,\infty)$, the conjugate of $p$ is denoted by $p^\prime$, i.e., $\frac{1}{p}+\frac{1}{p^\prime}=1.$
    \item  For any $f,g:[0,\infty)\rightarrow [0,\infty)$ we
denotes $f\asymp g$ if there exist constants $C_1,C_2>0$ such that $C_1f(t)\leq g(t)\leq C_2f(t)$ for all $t\in [0,\infty).$
\item $f^\prime$ represents the right derivatives of the function $f.$ 
\item For Young functions $\Phi$ and $\Psi$, their right derivatives are denoted by $\varphi$ and $\psi$, respectively.
\end{itemize}
 \subsection{Properties of Young function:}
In the following three propositions, we enlist some useful inequalities involving the Young functions.
\begin{prop}\cite{Krasn1961}
Let $\Phi$ be a Young function. Then the following inequalities hold for $t>0:$
 \begin{equation}\label{eqn-9}
     \Phi(t)\leq t\varphi(t)\leq \Phi(2t),
 \end{equation}
 \begin{equation}\label{eqn-10}
     \Phi\left(\frac{\tilde{\Phi}(t)}{t}\right)\leq \tilde{\Phi}(t),
 \end{equation}
 \begin{equation}\label{eqn-11}
    t\leq \Phi^{-1}(t)\tilde{\Phi}^{-1}(t)\leq 2t.
 \end{equation}
 In addition, if $\Phi\in \Delta_2,$ then 
 \begin{equation}\label{eqn-1}
     \Phi(t)\asymp t\varphi(t).
 \end{equation}
\end{prop}
\begin{prop}
Let $\Phi$ be a Young function, and $p^-_\Phi$ and $p^+_\Phi$ be as given in \eqref{*p-Phi*}.
\begin{enumerate}[(A)]
    \item If $\Phi,\,\tilde{\Phi}\in \Delta_2,$ then for $s,t>0$, the following inequalities hold for some $C\geq  1:$ 
\begin{align}
  \min\{s^{p^-_\Phi},s^{p^+_\Phi}\}\Phi(t)\leq \Phi(st)&\leq \max\{s^{p^-_\Phi},s^{p^+_\Phi}\}\Phi(t),\;\;\;\label{eqn-delta2-2}\\   \Phi^{-1}(st)&\leq \max\{s^{1/p^-_\Phi},s^{1/p^+_\Phi}\}\Phi^{-1}(t)\label{eqn----3}, 
  \\
    \tilde{\Phi}(st)&\leq \max\{s^{(p^-_\Phi)^\prime},s^{(p^+_\Phi)^\prime}\}\Phi(t),\;\;\;\label{eqn-delta2-2*}\\  \varphi(st)&\leq C \max\{s^{p^-_\Phi-1},s^{p^+_\Phi-1}\}\varphi(t).\label{eqn-delta2-2-2}
\end{align}   
 \item If $\Phi\in \Delta^\prime$, then for $s,t>0$, the following inequalities hold for some $C\geq 1:$
 \begin{align}
    \varphi(st)&\leq C\varphi(s)\varphi(t),\label{eqn-var}\\
     \Phi^{-1}(s)\Phi^{-1}(t)&\leq C\Phi^{-1}(st), \label{eqn-5}\\ \tilde{\Phi}^{-1}(st)&\leq  C \tilde{\Phi}^{-1}(s)\tilde{\Phi}^{-1}(t) \label{eqn-6}.
 \end{align}
  In addition, if $\tilde{\Phi}\in \Delta_2,$ then 
 \begin{equation}\label{eqn-7*}
     \tilde{\Phi}(s)\tilde{\Phi}(t)\leq C  \tilde{\Phi}(st).
 \end{equation}
\end{enumerate}
 \begin{proof} $(A):$  For proof of \eqref{eqn-delta2-2}, \eqref{eqn----3},  and \eqref{eqn-delta2-2*}, see \cite{Fukagai2006,Krasn1961}. From \eqref{eqn-1} and \eqref{eqn-delta2-2} we obtain \eqref{eqn-delta2-2-2}.

$(B):$ Inequality  \eqref{eqn-var} follows from \eqref{deld1} and \eqref{eqn-1}. For proof of \eqref{eqn-5}, see \cite[Page 6]{salort2021}.  Multiplying both sides of \eqref{eqn-5} by $\tilde{\Phi}^{-1}(s)\tilde{\Phi}^{-1}(t)\tilde{\Phi}^{-1}(st)$ and using \eqref{eqn-11} we obtain \eqref{eqn-6}. From \eqref{eqn-6} we get $\tilde{\Phi}^{-1}(\tilde{\Phi}(s)\tilde{\Phi}(t))\leq Cst$. Now, apply $\tilde{\Phi}$ on both sides and use \eqref{eqn-delta2-2}  to get \eqref{eqn-7*}.
\end{proof}
\end{prop}
\begin{prop}\label{lemcap}
Let $\Phi$ and $\Psi$ be Young functions such that  $\Phi,\tilde{\Psi}\in\Delta^\prime$ and $\Psi\in \Delta_2$. Then there exists  $C=C(\Phi,\Psi)\geq 1$ such that the following inequality holds for any $s,t\geq 0:$ $$\Psi\circ\Phi^{-1}(s) \Psi\circ\Phi^{-1}(t) \leq C\Psi\circ\Phi^{-1}(st).$$ 
\begin{proof}
 Since $\tilde{\Psi}\in \Delta^\prime$ and  $\Psi\in \Delta_2$, by \eqref{eqn-7*} there exists $C\geq 1$ such that  
\begin{equation*}\label{eqqd1}
    \tilde{\tilde{\Psi}}\circ\Phi^{-1}(s)\,  \tilde{\tilde{\Psi}}\circ\Phi^{-1}(t) \leq C \tilde{\tilde{\Psi}}\left(\Phi^{-1}(s)\Phi^{-1}(t)\right),
\end{equation*}
for all $s,t\geq 0$.
Moreover, by \eqref{eqn-5} there exists $C_1\geq 1$ such that  $\Phi^{-1}(s)\Phi^{-1}(t)\leq C_1\Phi^{-1}(st).$ Consequently,  using $ \tilde{\tilde{\Psi}}=\Psi$ and \eqref{eqn-delta2-2} we obtain
\begin{align*}
       \Psi\circ\Phi^{-1}(s) \Psi\circ\Phi^{-1}(t) \leq C\Psi\left(\Phi^{-1}(s)\Phi^{-1}(t)\right)\leq C\Psi\left(C_1\Phi^{-1}(st)\right)\leq CC_1^{p^+_\Psi}\Psi\circ\Phi^{-1}(st).
\end{align*}
\end{proof}
\end{prop}
\begin{exmp}\label{exam} The following Young functions satisfy the $\Delta^\prime$-condition and $\tilde{\Phi}\in \Delta_2$:
\begin{enumerate}[(i)]
    \item  $\Phi(t)=t^p,\;\;\;t\geq 0,\;p>1$;
    \item $\Phi(t)=t^p+t^q,\;\;\;t\geq 0,\; p,q>1$;
    \item $\Phi(t)=\max\{t^p,t^q\},\;\;\;t\geq 0,\;p,q>1$;
     \item $\Phi(t)=t^p\log (e+t),\;\;\; t\geq 0,\;p>1$;
    \item $\Phi(t)=t^p(1+|\log t|),\;\;\; t\geq 0,\;p>1$.
\end{enumerate}
\end{exmp}

\subsection{Symmetrization}
Let $\Omega\subset \mathbb{R}^N$ be an open set and $\mathcal{M}(\Omega)$ be the set of all extended real-valued Lebesgue measurable functions that are finite a.e. in $\Omega$. For $f\in \mathcal{M}(\Omega)$, we define the following notions:
\begin{itemize}
\item \textbf{One-dimensional decreasing rearrangement $f^*$:} For $t>0$,  $f^*$ is defined as $$f^*(t)=\inf \left\{s>0:|\{x\in \Omega:|f(x)|>s\}|<t\right\},$$ where $|E|$ denotes the Lebesgue measure of a set $E\subset\mathbb{R}^N$.
\item \textbf{Maximal function $f^{**}$:}  The \textit{maximal function $f^{**}$ of $f^*$} is defined as
\begin{equation}\label{maximal}
  f^{**}(t)=\frac{1}{t}\int_0^tf^*(\tau) d\tau,\;\;\; t>0.  
  \end{equation}
\end{itemize}
Next, we state two important inequalities related to symmetrization. 
\begin{prop}\label{prop1}
Let $\Phi$ be any Young function. Then the following inequalities hold.
\begin{enumerate}
    \item \textbf{Hardy-Littlewood inequality \cite{Edmunds-2004}:}  Let $u$ and $v$ be two measurable functions. Then $$\int_{\Omega}|u(x)|\,\Phi(|v(x)|) dx\leq \int_0^{|\Omega|}u^*(t)\Phi(v^*(t)) dt.$$
    \item  \textbf{P\'olya-Szeg\"o inequality \cite{Brothers1988}:}  If $u\in \mathcal{C}_c^1(\Omega)$, then 
        $$\int_0^{|\Omega|}\Phi\left(N\omega_N^{\frac{1}{N}}r^{1-\frac{1}{N}}\left(-\frac{du^*}{dr}\right)\right)  dr\leq \int_{\Omega}\Phi(|\nabla u(x)|)dx,$$
        where  $\omega_N$ is the measure of the unit ball in $\mathbb{R}^N$.
\end{enumerate}
\end{prop}
\subsection{Banach function space:} 
Recall that $\mathcal{M}(\Omega)$ is the set of all extended real-valued Lebesgue measurable functions that are finite a.e. in $\Omega$. Let $\mathcal{M}^+(\Omega)$ be the set of all non-negative functions in $\mathcal{M}(\Omega)$.
\begin{defn}\label{BFS-defn}
A {\it{Banach function norm}} is a map $\rho:\mathcal{M}^+(\Omega)\rightarrow[0,\infty]$, such that  for $f,\,g,\,f_n\,(n\in \mathbb{N})$ in $\mathcal{M}^+(\Omega)$, $\lambda\geq 0$, and  for Lebesgue measurable subsets $E$ of $\Omega$, the following are true:
\begin{enumerate}[(a)]
   \item $\rho(f)=0\iff f=0$ a.e.,\;\;\; $\rho(\lambda f)=\lambda \rho(f)$,\;\;\;$\rho(f+g)\leq \rho(f)+\rho(g);$ 
   \item if $ g\leq f$ a.e., then $\rho(g)\leq \rho(f);$
    \item if $ f_n \uparrow f$  a.e., then $\rho(f_n)\uparrow \rho(f);$   
  \item\label{p4} if $|E|< \infty$, then $\rho(\chi_E)<\infty;$
   \item\label{p5} if $|E|<\infty$, then there exists $C=C(|E|)\in (0,\infty)$ such that $\int_Ef dx\leq C\rho(f)$.
\end{enumerate}
 For a Banach function norm  $\rho$, the collection $R_\rho(\Omega)=\left\{ f\in \mathcal{M}(\Omega): \rho(|f|)<\infty\right\}$ is called a {\it{Banach function space}} with respect to norm $\|f\|_{R_\rho(\Omega)}:=\rho(|f|).$
 Indeed, $\left(R_\rho(\Omega),\|\cdot\|_{R_\rho(\Omega)}\right)$ is a Banach space.
 
 A Banach function norm $\rho$ is said to be  {\it{rearrangement-invariant}} if $\rho(f)=\rho(g)$ whenever $f,g\in\mathcal{M}^+(\Omega)$ are equimeasurable, i.e. $|\{x\in\Omega :|f(x)|>\lambda\}|=|\{x\in\Omega :|g(x)|>\lambda\}|$ for every $\lambda\geq 0$. The corresponding Banach function space is said to be a {\it{rearrangement-invariant Banach function space}}. For further readings on Banach function spaces, we refer to \cite{Edmunds-2004}.
 \end{defn}
 Next, we provide an explicit construction of certain Banach function spaces that appear in this article.

\begin{prop}\label{B-F-S}
Let $\eta$ be a positive function on $(0,|\Omega|)$ such that 
\begin{enumerate}[(i)]
    \item  $\eta$ is bounded on $(0,r]$ for all $r\in (0,|\Omega|),$
    \item $\frac{\eta(r)}{r}$ is decreasing on $(0,|\Omega|)$.
\end{enumerate}
For $f\in \mathcal{M}^+(\Omega)$, define  $$\rho_\eta(f)=\sup_{0<r<|\Omega|}\{f^{**}(r)\eta(r)\}.$$
Then $\rho_\eta$ is a Banach function norm, and the space $M_\eta(\Omega)=\{f\in \mathcal{M}(\Omega):\rho_\eta(|f|)<\infty\}$ is a rearrangement-invariant Banach function space.
\end{prop}
\begin{proof}
 The conditions $\mathrm{(a)}$, $\mathrm{(b)}$, and $\mathrm{(c)}$ of Definition \ref{BFS-defn} follow from the elementary properties of $f^{**}$ (see \cite[Proposition 3.2.15, Theorem 3.2.16]{Edmunds-2004}). To verify $\mathrm{(d)}$ of Definition \ref{BFS-defn}, let $E\subset\Omega$ with measure  $|E|=r<\infty$. Then $\chi_E^*=\chi_{[0,r)}$ and so
  $$ \rho_\eta(\chi_E)=\sup_{0<s<|\Omega|}\{\chi_{E}^{**}(s)\eta(s)\}=\sup_{0<s<|\Omega|}\left\{\min\left(1,\frac{r}{s}\right)\eta(s)\right\}$$ $$=\max\left\{\sup_{0<s<r}\eta(s),\;\;\; r\cdot\sup_{r\leq s<|\Omega|}\frac{\eta(s)}{s}\right\}=\sup_{0<s\le r}\eta(s), $$ where the last equality is obtained from the assumption, $\eta(s)/s$ is decreasing on $(0,|\Omega|)$. Now $(d)$ follows from the assumption $(i).$
 Finally, for $(e)$ we consider $f\in M_\eta(\Omega)$ and a Lebesgue measurable subset $E$  of $\Omega$ with $|E|=r\in (0,\infty)$. By the Hardy-Littlewood principle, we have
    $$  \left|\int_Ef(x)dx\right|\leq \int_0^{r}f^*(s)ds=rf^{**}(r)\leq\frac{r}{\eta(r)}\cdot\sup_{0<s<|\Omega|}\{f^{**}(s)\eta(s)\}=C\rho_\eta(f),
$$
where  $C=\frac{r}{\eta(r)}<\infty$. This proves $\mathrm{(e)}$ of Definition \ref{BFS-defn}. Since $\rho_\eta$ is defined in terms of $f^{**}$, we conclude that $M_\eta(\Omega)$ is a rearrangement-invariant Banach function space. This completes the proof of the proposition.
\end{proof}

\begin{rem}\label{BF-S}
\begin{enumerate}[(i)]
    \item 

Let $\Phi$ be a Young function such that $p^+_\Phi<N$. For $r>0$, consider $\eta(r)=1/{\Phi\left(r^{-\frac{1}{N}}\right)}$. It is easy to check that $\eta$ is positive on $ (0,|\Omega|)$ and bounded on $(0,r]$ for all $r\in (0,|\Omega|)$. Observe that 
 $$\left(\frac {\eta(r)}{r}\right)^\prime=\frac{1}{N}\left\{r^{-\frac{1}{N}}\varphi\left(r^{-\frac{1}{N}}\right)-N\Phi\left(r^{-\frac{1}{N}}\right)\right\}\left(r\Phi\left(r^{-\frac{1}{N}}\right)\right)^{-2},\;\;\;\,r\in (0,|\Omega|).$$ Moreover, the definition of $p_\Phi^+$ (see \eqref{*p-Phi*}) gives  $r^{-\frac{1}{N}}\varphi\left(r^{-\frac{1}{N}}\right)\leq p^+_\Phi\Phi\left(r^{-\frac{1}{N}}\right)<N\Phi\left(r^{-\frac{1}{N}}\right)$. Consequently  $\left(\eta(r)/r\right)^\prime<0$, and hence $\eta(r)/r$ is a decreasing function. Now, by the above proposition, $L^{\Phi,\infty}(\Omega)$ (see \eqref{Phi-infty}) is a rearrangement-invariant Banach function space.
 
\item For a Young function $\Phi$ satisfying \eqref{8} and $Q_\Phi$ as given in \eqref{eqA*}, consider $\eta(r)=r\varphi\left( \int_r^{|\Omega|}\frac{1}{Q_\Phi(s)}ds\right)$.  Clearly, $\eta(r)> 0$ for all $r\in(0,|\Omega|)$ and $\eta(r)/r$ decreases on $(0,|\Omega|)$. Since $\Phi$ satisfies \eqref{8}, $\eta$ is bounded on $(0,r]$ for all $r\in (0,|\Omega|)$. Thus, it follows from Proposition \ref{B-F-S} that $X_\Phi(\Omega)$ (see \eqref{Phi-Phi}) is rearrangement-invariant Banach function space.

 \item Let $\Phi$ and $\Psi$ be Young functions satisfying \eqref{7} condition. Consider the function $\eta(r)=r\Psi\left(\|\zeta\|_{L^{\tilde{\Phi}}((r,|\Omega|))}\right)$, where $\zeta(s)=s^{\frac{1}{N}-1}$. Notice that,  $\eta$ is positive on $(0,|\Omega|)$ and $\eta(r)/r$ decreases on $(0,|\Omega|).$ 
Using \eqref{7} we can see that $\eta$ is bounded on $(0,r]$ for all $r\in (0,|\Omega|)$.  Therefore, by Proposition \ref{B-F-S}, $X_{\Phi,\Psi}(\Omega)$ (see \eqref{Phi-Psi}) is a rearrangement-invariant Banach function space.
\end{enumerate}
\end{rem}

\subsection{Other function spaces:} Here, we briefly discuss some function spaces that are needed for the development of this article.
\begin{enumerate}[(A)]
    \item \textbf{ Weighted Orlicz spaces:}\label{orlicz-sob}
Given a Young function $\Phi$ satisfying the $\Delta_2$-condition, an open set $\Omega\subset\mathbb{R}^N$, and $g\in \mathcal{M}^+(\Omega),$ we define
$$ L^{\Phi,g}(\Omega)=\left\{u:\Omega\rightarrow \mathbb{R}\; \text{measurable}\; :\int_{\Omega}\Phi\left( \frac{|u(x)|}{\lambda}\right)g(x) dx<\infty\; \text{for some}\; \lambda>0\right\}.$$
 The space $L^{\Phi,g}(\Omega)$  is a Banach space with respect to the following Luxemburg norm:
\begin{equation}\label{defn-Luxemberg}
   \|u\|_{L^{\Phi,g}(\Omega)}=\inf\left\{ \lambda>0:\int_{\Omega}\Phi\left( \frac{|u(x)|}{\lambda}\right)g(x) dx\leq 1\right\}.
\end{equation}
If $g\equiv 1$, $L^{\Phi,g}(\Omega)$ coincides with the usual Orlicz space, it is denoted by $L^\Phi(\Omega)$. In particular, for $\Phi(t) = t^p$ with $p \in (1, \infty)$, $L^\Phi(\Omega) = L^p(\Omega)$  and $\|u\|_{L^{\Phi}(\Omega)}=\|u\|_{L^{p}(\Omega)}$.
\item \textbf{Orlicz-Sobolev spaces:} The Orlicz-Sobolev space is defined by
$$ W^{1,\Phi}(\Omega)=\left\{u\in L^\Phi(\Omega) :\;\;|\nabla u|\in L^\Phi(\Omega) \right\},$$
where $\nabla u$ is considered in the distributional sense. The space $ W^{1,\Phi}(\Omega)$ is a reflexive Banach space with the norm $\|u\|_{L^{1,\Phi}(\Omega)}:=\|u\|_{L^\Phi(\Omega)}+\|\nabla u\|_{L^\Phi(\Omega)}$  when $\Phi,\tilde{\Phi}\in \Delta_2$.
\end{enumerate}
For further readings on Orlicz and Orlicz-Sobolev spaces, we refer to \cite{Adams1975,Lars2011,Krasn1961}. 

In the following proposition, we list some properties of the Luxemburg norm.
 \begin{prop}\cite[Lemma 2.1.14,
Lemma 2.6.5]{Lars2011}\label{prop2-pell}
 Let $\Omega$ be an open subset of $\mathbb{R}^N$,  $\Phi$ be a Young function, and $g\in \mathcal{M}^+(\Omega)$.  If $u\in L^{\Phi,g}(\Omega)$ and $ v\in L^{\tilde{\Phi},g}(\Omega)$, then the following hold: 
 \begin{enumerate}[(i)]
     \item  H\"older's inequality: $\displaystyle\int_{\Omega}|u(x)v(x)|g(x) dx\leq 2\|u\|_{L^{\Phi,g}(\Omega)}\|v\|_{L^{\tilde{\Phi},g}(\Omega)}.$
     \item $  \|u\|_{L^{\Phi,g}(\Omega)}\leq 1+\displaystyle\int_{\Omega}\Phi(|u(x)|)g(x)dx.$
     \item $ \displaystyle\int_{\Omega}\Phi({\beta|u(x)|})g(x) dx=1, $ where $\beta^{-1}=\|u\|_{L^{\Phi,g}(\Omega)}.$
  \end{enumerate}
  \end{prop}

\subsection{Some embedding results}
First, we state a necessary and sufficient condition for the continuous embedding between the weighted Orlicz spaces. The proof follows similarly as the proof of \cite[Theorem 8.12]{Adams1975}.
\begin{thm}\label{emb-ele}
 Let $\Omega$ be an open subset of $\mathbb{R}^N$ and $g\in L^1_{\mathrm{loc}}(\Omega)$ with $g\geq 0$. Let $\Phi$ and $\Psi$ be Young functions satisfying the $\Delta_2$-condition. Then  $L^{\Phi,g}(\Omega)\hookrightarrow L^{\Psi,g}(\Omega)$  if and only if one of the following conditions holds:
    \begin{enumerate}[(i)]
        \item $\Phi$ dominates $\Psi$ globally,
        \item $\Phi$ dominates $\Psi$ near infinity and $\Omega$ is bounded.
    \end{enumerate}
\end{thm}
Next, we state an embedding theorem due to  Cianchi; see \cite[Theorem 1, Theorem 3]{cianchi1996}, \cite[Theorem 1]{Cianchi2000}, and   \cite[Theorem 1a, Corollary 1]{Cianchi2-1996}.
\begin{thm}\label{thm-cianchi}
 Let $\Omega$ be an open subset of $\mathbb{R}^N$, and  $\Phi,$ $\Psi$ be  Young functions. Assume that $\Phi$ satisfies \eqref{H1} when $|\Omega|=\infty.$ We have
 \begin{enumerate}[(i)]
     \item  $\mathcal{D}_0^{1,\Phi}(\Omega)\hookrightarrow L^{\Phi_N}(\Omega),$ where $\Phi_N$ is given in \eqref{Phi--N}. 

     \item if $\Phi$ satisfies \eqref{eqG1}, then $\mathcal{D}_0^{1,\Phi}(\Omega)\hookrightarrow L^\infty(\Omega)$.

     \item for a bounded Lipschitz domain $\Omega,$ the following embeddings are compact:
  \[W^{1,\Phi}(\Omega)\hookrightarrow
     \begin{cases} 
       \text{$L^{\Psi}(\Omega)$} &\quad \text{if  $\Psi\prec\prec \Phi_N$,}\\
       \text{$L^\infty(\Omega)$} & \quad\text{if  $\Phi$ satisfies \eqref{eqG1}.}
     \end{cases}
\]
 \end{enumerate}
 
\end{thm}
\begin{prop}\label{prop:comp}
Let $\Omega$ be an open subset of $\mathbb{R}^N$, and $\Phi$ and $\Psi$ be Young functions such that $\Phi,\tilde{\Phi},\Psi\in\Delta_2.$  Further assume that $\Phi$ satisfies \eqref{H1} when $|\Omega|=\infty.$ Let  one of the following conditions hold:
\begin{enumerate}[(i)]
\item  $\Psi\prec\prec\Phi_N$,
\item  $\Phi$ satisfies \eqref{eqG1}.
\end{enumerate}
 Then the embedding $\mathcal{D}_0^{1,\Phi}(\Omega)\hookrightarrow L_\mathrm{loc}^{\Psi}(\Omega)$ is compact.
\begin{proof}  
Since $\mathcal{D}_0^{1,\Phi}(\Omega)$ a is reflexive space (see \cite[Proposition 3.1]{Cianchi2017}), it is equivalent to show that, for a bounded Lipschitz domain $\Omega_1\subset \Omega$  and a bounded sequence $(u_n)$ in $\mathcal{D}_0^{1,\Phi}(\Omega)$,  there exists a subsequence of $(u_n|_{\Omega_1})$ that converges in $L^{\Psi}(\Omega_1)$.
\begin{enumerate}[$(i)$]
     \item In this case, by  Theorem \ref{thm-cianchi} we have $\mathcal{D}_0^{1,\Phi}(\Omega)\hookrightarrow L^{\Phi_N}(\Omega)$. Since  $\Phi\prec\prec\Phi_N$ (see \cite[Page 133]{Rudd2005}) by Theorem \ref{emb-ele}, we also have $L^{\Phi_N}(\Omega_1)\hookrightarrow L^{\Phi}(\Omega_1)$. Therefore, $(u_n|_{\Omega_1})$ is a bounded sequence  in $W^{1,\Phi}(\Omega_1).$ Now  Theorem \ref{thm-cianchi} assures that $(u_n|_{\Omega_1})$ has a convergent sub-sequence in $L^{\Psi}(\Omega_1)$. 

     \item In this case, we have   $\mathcal{D}_0^{1,\Phi}(\Omega)\hookrightarrow L^{\infty}(\Omega)$ (by Theorem \ref{thm-cianchi}), and  $L^\infty(\Omega_1)\hookrightarrow L^{\Phi}(\Omega_1)$. Thus, 
     $(u_n|_{\Omega_1})$ is a bounded sequence  in $W^{1,\Phi}(\Omega_1).$ Since $W^{1,\Phi}(\Omega_1)\hookrightarrow L^\infty(\Omega_1)$ is compact (Theorem \ref{thm-cianchi}) and $L^\infty(\Omega_1)\hookrightarrow L^{\Psi}(\Omega_1)$, we conclude that $(u_n|_{\Omega_1})$ has a convergent sub-sequence in $L^{\Psi}(\Omega_1)$. 

\end{enumerate}
\end{proof}
\end{prop}
If $\Psi\in \Delta_2,$ from  
\eqref{modular}, it is easy to deduce the following
 {\it{weighted norm inequality}:}
 \begin{equation}\label{norm-ineq}
 \|u\|_{L^{\Psi,|g|}(\Omega)}\leq C_1\|\nabla u\|_{L^\Phi(\Omega)},\;\;\;\forall\,u\in\mathcal{C}_c^1(\Omega), \end{equation}
  for some $C_1>0$. However, the converse may not be true for a general $\Phi.$ The following lemma ensures that under some assumptions on $\Phi$ and $\Psi$, one can obtain \eqref{modular} from \eqref{norm-ineq}.
 \begin{lem}\label{rem8*}
 Let $\Phi$ and $\Psi$ be Young functions such that $\Phi\in\Delta_2$, $\,\tilde{\Phi},\,\Psi\in\Delta^\prime$, and $g\in L^1_\mathrm{loc}(\Omega)$. If \eqref{norm-ineq} holds, then there exists $C_2=C_2(\Phi,\Psi)>0$  such that \eqref{modular} holds with $C=C_2C_1$.
 \begin{proof}
Let $\Phi,\Psi$, and $g$ be as given above. Since $\Psi\in \Delta^\prime$, by \eqref{deld1} there exists $C_3\geq 1$ such that
\begin{equation}\label{eqn1-lema2.13}
    \Psi\left(| u(x)|\right)\leq C_3\Psi\left(C_1\|\nabla u\|_{L^\Phi(\Omega)}\right)\Psi\left(\frac{| u(x)|}{C_1\|\nabla u\|_{L^ \Phi(\Omega)}}\right),\;\;\;\forall\,u\in \mathcal{C}_c^1(\Omega),\, x\in \Omega.
\end{equation}
Now use \eqref{eqn1-lema2.13},  \eqref{norm-ineq}, and $(iii)$ of Proposition \ref{prop2-pell} to get
\begin{align*}
    \int_\Omega \Psi(|u(x)|) |g(x)|dx&\leq  C_3\Psi\left(C_1\|\nabla u\|_{L^\Phi(\Omega)}\right) \int_{\Omega}\Psi\left(\frac{| u(x)|}{C_1\|\nabla u\|_{L^\Phi(\Omega)}}\right)|g(x)|dx\\&\leq  C_3\Psi\left(C_1\|\nabla u\|_{L^\Phi(\Omega)}\right)\int_{\Omega}\Psi \left(\frac{| u(x)|}{\| u\|_{L^{ \Psi,|g|}(\Omega)}} \right) |g(x)|dx\\&=C_3\Psi\left(C_1\|\nabla u\|_{L^\Phi(\Omega)}\right).
    \end{align*}  Applying $\Psi^{-1}$ on both sides of the above inequality and using \eqref{eqn----3}, we get 
\begin{equation*}
    \Psi^{-1}\left( \int_\Omega \Psi(|u(x)|) |g(x)|dx\right)\leq C_3^{1/p^-_\Psi}C_1\|\nabla u\|_{L^\Phi(\Omega)}.
\end{equation*}
Thus, the proof is complete if we show that \begin{equation}\label{eqn2-lem3.1}
   \|\nabla u\|_{L^\Phi(\Omega)}\leq C_4 \Phi^{-1}\left( \int_\Omega \Phi(|\nabla u(x)|) dx\right) ,\;\;\;\forall\,u\in \mathcal{C}_c^1(\Omega),
\end{equation}
for some $C_4>0$. Since $\Phi\in \Delta_2$ and $\tilde{\Phi}\in \Delta^\prime$, by \eqref{eqn-7*} there exists $C_5\geq 1$ such that
$$\Phi\left(\|\nabla u\|_{L^\Phi(\Omega)}\right)\Phi\left(\frac{|u(x)|}{\|\nabla u\|_{L^\Phi(\Omega)}}\right)\leq C_5\Phi\left(|\nabla u(x)|\right),\;\;\;\forall\,u\in \mathcal{C}_c^1(\Omega),\, x\in \Omega.$$ Integrate both sides of the above inequality over $\Omega$ and use  $(iii)$ of Proposition \ref{prop2-pell} to yield  
\begin{equation*}
   \Phi\left(\|\nabla u\|_{L^\Phi(\Omega)}\right)=\Phi\left(\|\nabla u\|_{L^\Phi(\Omega)}\right) \int_\Omega\Phi\left(\frac{|u(x)|}{\|\nabla u\|_{L^\Phi(\Omega)}}\right)dx\leq  C_5\int_\Omega\Phi\left(|\nabla u(x)|\right)dx.
\end{equation*}
Now apply $\Phi^{-1}$ on both sides of the above inequality and use \eqref{eqn----3} to get  \eqref{eqn2-lem3.1}. This completes the proof.
\end{proof}
 \end{lem}

  \subsection{Muckenhoupt condition}
  We recall the Muckenhoupt type necessary and sufficient condition involving the Young function obtained by Lai, see \cite[Theorem 5]{LQ}. For further readings on these inequalities, we refer to \cite[Chapter 11]{kufner2007}.

  \begin{prop}\label{mucken2}
Let $b\in (0,\infty]$  and $\Phi,\Psi$ be Young functions such that $\Psi\in \Delta_2$  and $\Psi\circ\Phi^{-1}$ is super-additive. 
Let $w,v$ be non-negative locally integrable functions on $(0,b)$ with $v>0$. Then 
 \begin{equation}\label{eq18}
\Psi^{-1}\left(\int_0^b\Psi\left(\left|\int_t^bf(s) ds\right|\right)w(t) dt \right)\leq C\Phi^{-1}\left(\int_0^b\Phi(|f(t)|)v(t) dt\right)
 \end{equation} 
holds for all measurable function $f$ on $(0,b)$ if and only if 
 \begin{equation}\label{eq19}
     \Psi^{-1}\left(\Psi\left(\frac{1}{\epsilon }\left\|1/v\right\|_{L^{\tilde{\Phi},\epsilon v}((r,b))}\right) \int_0^rw(t)dt\right)\leq D\Phi^{-1}\left(\frac{1}{\epsilon}\right),
 \end{equation}
holds for all $\epsilon>0$ and for all $r\in (0, b)$.
 
 In addition, for
the best constants $C$ and $D$, there exist positive constants $\alpha_1,\,\alpha_2,$ and $M$ depending only on $\Phi$ and $\Psi$ such that $D\leq C\leq M\max\{D^{\alpha_1},D^{\alpha_2}\}$.
 \end{prop}

\section{Admissible function spaces in $\mathcal{H}_{\Phi,\Phi}(\Omega).$}
In this section, we prove some important propositions, Theorem \ref{thm:3}, Theorem \ref{larentzthm}, and Theorem \ref{cor}. Recall that
\begin{equation*}
     \Phi_N(t)= \int_0^ts^{N^\prime-1}\left(H_\Phi^{-1}\left(s^{N^\prime}\right)\right)^{N^\prime}  ds\; \;\; \text{for}\;t\geq 0, 
      \end{equation*}
 where $H_\Phi^{-1}$ is the inverse of $H_\Phi(t)=\displaystyle\int_0^t\frac{\tilde{\Phi}(s)}{s^{1+N^\prime}}ds$. In the following proposition, we provide a sufficient condition on $\Phi$ so that $B_\Phi=\Phi_N\circ\Phi^{-1}$ is a Young function.
\begin{prop}\label{B-Phi-Young} Let $\Phi$ be a Young function such that $\frac{t\varphi^\prime(t)}{\varphi(t)}\leq \frac{Np^-_\Phi}{N-p^-_\Phi}-1$ holds for all $t>0$. Then $B_\Phi$ is a Young function. 
\begin{proof} To prove $B_\Phi$ is a Young function, it is enough to show that 
$B_\Phi^{\prime\prime}(t)\geq 0$ (\cite[Theorem 1.1]{Krasn1961}). By direct computations, we easily obtain $$B_\Phi^{\prime\prime}(t)=\frac{\Phi_N^{\prime\prime}(s)\varphi(s)-\Phi_N^{\prime}(s)\varphi^\prime(s)}{\varphi(s)^3},\;\;\;\text{where}\,\,s=\Phi^{-1}(t).$$ Thus, $$B_\Phi^{\prime\prime}(t)\geq 0,\; \forall\, t> 0 \text{ if and only if } \dfrac{\Phi_N^{\prime\prime}(t)}{ \Phi_N^{\prime}(t)}\geq \dfrac{\varphi^\prime(t)}{\varphi(t)},\,\, \forall\,t>0.$$
It is not difficult to see that
\begin{align*}
    \Phi_N^\prime(t)&=t^{N^\prime-1}\left(H_\Phi^{-1}(t^{N^\prime})\right)^{N^\prime},\\\Phi_N^{\prime\prime}(t)&=(N^\prime-1)t^{N^\prime-2}\left(H_\Phi^{-1}(t^{N^\prime})\right)^{N^\prime}+\left(N^\prime t^{N^\prime-1}\right)^2\frac{\left(H_\Phi^{-1}(t^{N^\prime})\right)^{N^\prime-1}}{H_\Phi^\prime\left(H_\Phi^{-1}(t^{N^\prime})\right)}.
\end{align*} Therefore, 
\begin{equation}\label{eqn1-B}
    \frac{\Phi_N^{\prime\prime}(t)}{\Phi_N^{\prime}(t)}=\frac{N^\prime-1}{t}+\frac{(N^\prime)^2 t^{N^\prime-1}}{H_\Phi^\prime\left(H_\Phi^{-1}(t^{N^\prime})\right)H_\Phi^{-1}(t^{N^\prime})}.
\end{equation} 
Since $(p^-_\Phi)^\prime\geq (p^+_\Phi)^\prime,$ using \eqref{eqn-delta2-2*} we get 
 \begin{equation*}
    \frac{\tilde{\Phi}(sr)}{r^{1+N^\prime}}\leq s^{(p^-_\Phi)^\prime}\frac{\tilde{\Phi}(r)}{r^{1+N^\prime}},\;\;\;\forall\,s> 1,\,\forall\,r> 0.
 \end{equation*}
 Next, we integrate both sides of the above inequality over  $(0,\tau)$ with respect to $r$ and  use the definition of $H_\Phi$ to get  
\begin{equation*}
 s^{N^\prime}H_\Phi(s\tau)\leq s^{(p^-_\Phi)^\prime} H_\Phi(\tau),\;\;\;\forall\,s> 1,\,\forall\,\tau> 0.
\end{equation*}
From the above inequality, we can deduce that
$$\frac{H_\Phi(s\tau)-H_\Phi(\tau)}{s-1}\leq \frac{s^{(p^-_\Phi)^\prime-N^\prime}-1}{s-1}H_\Phi(\tau),\;\;\;\forall\,s> 1,\,\forall\,\tau>0.$$  Letting $s\to 1$  we have $$ \tau H_\Phi^\prime(\tau)\leq \left((p^-_\Phi)^\prime-N^\prime\right)H_\Phi(\tau),\;\;\;\forall\,\tau>0.$$ 
Now  using the above inequality with $\tau=H_\Phi^{-1}(t^{N^\prime})$ in \eqref{eqn1-B} we obtain
\begin{equation*}
    \frac{\Phi_N^{\prime\prime}(t)}{\Phi_N^{\prime}(t)}\geq \frac{N^\prime-1}{t}+\frac{(N^\prime)^2 }{t((p^-_\Phi)^\prime-N^\prime)}=\frac{1}{t}\left(\frac{Np^-_\Phi}{N-p^-_\Phi}-1\right),\;\;\;\forall\,t>0. 
\end{equation*} 
Hence,  using $\frac{t\varphi^\prime(t)}{\varphi(t)}\leq \frac{Np^-_\Phi}{N-p^-_\Phi}-1$ we get $B_\Phi^{\prime\prime}\geq 0$, and so the proof is complete.
\end{proof}
\end{prop}
\begin{rem}
    Let $\Phi_*$ be the {\it{Sobolev conjugate}} of $\Phi$ (\cite[Page 248]{Adams1975}). Under the similar assumptions as given in proposition \ref{B-Phi-Young}, it is known that $\Phi^*\circ\Phi^{-1}$ is a Young function, see \cite[Lemma 8]{salort2023}.
 \end{rem}
\begin{prop}\label{prop-main*}
Let $\Omega$ be an open subset of $\mathbb{R}^N$, and $\Phi$ be a Young function such that  $B_\Phi$ is a Young function and $\Phi\in\Delta^\prime$. Assume that $\Phi$ satisfies \eqref{H1} when $|\Omega|=\infty$. Then there exists  $C=C(N,\Phi)>0$  so that $$\|\Phi(|u|)\|_{L^{B_\Phi}(\Omega)}\leq C\Phi\left(\|\nabla u\|_{L^\Phi(\Omega)}\right), \;\;\; \forall\,u\in \mathcal{C}_c^1(\Omega).$$ 
\begin{proof}
 Let $C\geq 1$ be as given in  \eqref{eqn-5}. Now for $u\in \mathcal{C}_c^1(\Omega)$ and  $x\in \Omega$, we have
\begin{align*} 
     \Phi^{-1}\left(\frac{\Phi(|u(x)|)}{\Phi\left(C\|u\|_{L^{\Phi_N}(\Omega)}\right)} \right)&=\frac{\Phi^{-1}(\Phi(C\|u\|_{L^{\Phi_N}(\Omega)}))}{C\|u\|_{L^{\Phi_N}(\Omega)}} \Phi^{-1}\left(\frac{\Phi(|u(x)|)}{\Phi(C\|u\|_{L^{\Phi_N}(\Omega)})} \right)\\&\leq \frac{1}{\|u\|_{L^{\Phi_N}(\Omega)}} \Phi^{-1}\left(\Phi(C\|u\|_{L^{\Phi_N}(\Omega)})\frac{\Phi(|u(x)|)}{\Phi(C\|u\|_{L^{\Phi_N}(\Omega)})} \right)=\frac{|u(x)|}{\|u\|_{L^{\Phi_N}(\Omega)}},
 \end{align*}
where the inequality follows from \eqref{eqn-5}. Therefore, by noting  $B_\Phi=\Phi_N\circ\Phi^{-1}$ is a Young function and using $(iii)$ of Proposition \ref{prop2-pell} we get
\begin{equation*}
  \int_{\Omega} B_\Phi\left(\frac{\Phi(|u(x)|)}{\Phi(C\|u\|_{L^{\Phi_N}(\Omega)})} \right)dx\leq \int_{\Omega} \Phi_N\left(\frac{|u(x)|}{\|u\|_{L^{\Phi_N}(\Omega)}} \right)dx=1.
  \end{equation*}
 This gives
 $\|\Phi(|u|)\|_{B_\Phi}\leq \Phi\left(C\|u\|_{L^{\Phi_N}(\Omega)}\right).$
Moreover, by Theorem \ref{thm-cianchi}  we get
$$\|u\|_{L^{\Phi_N}(\Omega)} \leq  C_1\|\nabla u\|_{L^\Phi(\Omega)},\;\;\;\forall\,u\in \mathcal{C}_c^1(\Omega),$$ for some  $C_1>0.$  Consequently, using \eqref{eqn-delta2-2}
 we obtained the required inequality: $$\|\Phi(|u|)\|_{B_\Phi}\leq \Phi\left(CC_1\|\nabla u\|_{L^\Phi(\Omega)}\right)\leq \max\left\{(CC_1)^{p^-},(CC_1)^{p^+}\right\}\Phi\left(\|\nabla u\|_{L^\Phi(\Omega)}\right).$$ 
    \end{proof}
\end{prop}

\noindent\textbf{Proof of Theorem \ref{thm:3}:} Let $u\in \mathcal{C}^1_c(\Omega)$ and $g\in L^{\tilde{B}_\Phi}(\Omega).$
Then, the H\"older's inequality and Proposition \ref{prop-main*} gives
 \begin{equation*}
     \int_{\Omega}|g(x)|\,\Phi(|u(x)|)dx\leq 2\|g\|_{L^{\tilde{B}_\Phi}(\Omega)}\|\Phi(|u|)\|_{L^{B_\Phi}(\Omega)}\leq 2C\|g\|_{L^{\tilde{B}_\Phi}(\Omega)}\Phi\left(\|\nabla u\|_{L^\Phi(\Omega)}\right).
 \end{equation*}
 Replacing $u$ by $u/\|u\|_{L^{\Phi,g}(\Omega)}$ in the above  inequality and using $(iii)$ of Proposition \ref{prop2-pell} we get $$1\leq 2C \|g\|_{L^{\tilde{B}_\Phi}(\Omega)} \Phi\left(\|\nabla u\|_{L^\Phi(\Omega)}/\|u\|_{L^{\Phi,|g|}(\Omega)}\right).$$
Now apply $\Phi^{-1}$ on both sides of the above inequality and use  \eqref{eqn-6} for $\tilde{\Phi}$ to get 
\begin{align*}
  \|u\|_{L^{\Phi,|g|}(\Omega)}\leq C_1\Phi^{-1}\left(\|g\|_{L^{\tilde{B}_\Phi}(\Omega)}\right)\|\nabla u\|_{L^\Phi(\Omega)}
\end{align*}
for some $C_1>0$.
Therefore, by Lemma \ref{rem8*}, there exists $C_2>0$ such that
\begin{equation*}
    \Phi^{-1}\left(\int_{\Omega}|g(x)|\,\Phi(|u(x)|)dx \right)\leq C_2C_1\Phi^{-1}\left(\|g\|_{L^{\tilde{B}_\Phi}(\Omega)}\right)\Phi^{-1}\left(\int_{\Omega}\Phi(|\nabla u(x)|) dx \right).
\end{equation*}
Applying $\Phi$ on both sides of the above inequality and using \eqref{deld1} we obtain 
\begin{align*}
   \int_{\Omega}|g(x)|\,\Phi(|u(x)|)dx &\leq  C_3\|g\|_{L^{\tilde{B}_\Phi}(\Omega)}\int_{\Omega}\Phi(|\nabla u(x)|) dx,
\end{align*}
for some $C_3>0$. This completes the proof.
\qed

We prove the following proposition before giving the proof for Theorem \ref{larentzthm}.
\begin{prop}\label{prop-1.3} Let $\Omega$ be an open subset of $\mathbb{R}^N$, and $\Phi$ be a Young function satisfying $\Delta_2$-condition and $p^+_\Phi<N$. Then there exists $C=C(N,\Phi)>0$  so that 
\begin{equation*}
    \int_0^{|\Omega|}\Phi\left(s^{-\frac{1}{N}}u^*(s)\right)ds\leq C\int_{\Omega}\Phi(|\nabla u(x)|)dx,\;\;\;\forall\,u\in \mathcal{C}_c^1(\Omega).
\end{equation*}
\begin{proof}
By \eqref{eqn-delta2-2}, we obtain $\Phi(s)\leq s^{p^+_\Phi}\Phi(1)$ for all $s\in [1,\infty)$ and  $\Phi(s)\geq s^{p^+_\Phi}\Phi(1)$ for all $s\in (0,1)$. Now using  $N>p^+_\Phi$, we conclude that
\begin{equation*}
    \int_{1}^{\infty}\left(\frac{s}{\Phi(s)}\right)^{1/(N-1)}ds = \infty,\;\;\;
    \int_{0}^{1}\left(\frac{s}{\Phi(s)}\right)^{1/(N-1)}ds < \infty.
\end{equation*}
Moreover, from \eqref{eqn-delta2-2} we get that 
the upper Matuszewska-Orlicz index (see \cite[equation 1.24]{Cianchi2004})
\begin{equation*}\label{index}
      I(\Phi)=\lim_{t\to \infty}\frac{\log\left(\sup_{s>0}\frac{\Phi(t s)}{\Phi(s)}\right)}{\log t}\leq p^+_\Phi<N.
      \end{equation*}
Thus, by \cite[Remark 1.2]{Cianchi2004} and \cite[(I) of Proposition 5.2]{Cianchi2004}, there exists $C>0$ such that
\begin{equation*}
      \int_0^{|\Omega|}\Phi\left(s^{-\frac{1}{N}}u^*(s)\right)ds\leq \int_{\Omega}\Phi(C|\nabla u(x)|)dx,\;\;\;\,\forall\,u\in \mathcal{C}_c^1(\Omega).
  \end{equation*}
  Hence, the result follows from \eqref{eqn-delta2-2}.
\end{proof}
\end{prop}

 \noindent\textbf{Proof of Theorem \ref{larentzthm}:} 
Let $u\in \mathcal{C}^1_c(\Omega).$  Since $\tilde{\Phi}\in \Delta^\prime$, by  \eqref{eqn-7*} there exists $C\geq 1$ such that  
     \begin{equation*}\label{eqn::8}
       \Phi(s^{-\frac{1}{N}})\Phi\left(u^*(s)\right)\leq C\Phi(s^{-\frac{1}{N}}u^*(s)),\;\;\;\forall\,s\in(0,|\Omega|).
     \end{equation*}
Now using  Proposition \ref{prop1} and $g^*(s)\leq g^{**}(s)$, we get
\begin{align*}
\int_{\Omega}|g(x)|\,\Phi(|u(x)|)dx&\leq\int_0^{|\Omega|}g^{*}(s)\Phi(u^*(s))ds\nonumber= \int_0^{|\Omega|}\frac{g^{*}(s)}{\Phi(s^{-\frac{1}{N}})}\Phi(s^{-\frac{1}{N}})\Phi(u^*(s))ds\nonumber\\&\leq  C\sup_{0<s<|\Omega|}\left\{\frac{g^{**}(s)}{\Phi(s^{-\frac{1}{N}})}\right\}\int_0^{|\Omega|}\Phi(s^{-\frac{1}{N}}u^*(s))ds\\&= C \|g\|_{L^{\Phi,\infty}(\Omega)}\int_0^{|\Omega|}\Phi(s^{-\frac{1}{N}}u^*(s))ds.
   \end{align*}
Thus, by Proposition \ref{prop-1.3} we obtain
\begin{align*}
      \int_{\Omega}|g(x)|\,\Phi(|u(x)|)dx&\leq CC_1 \|g\|_{L^{\Phi,\infty}(\Omega)}\int_{\Omega}\Phi\left(|\nabla u(x)|\right)dx,
\end{align*}
for some $C_1>0$. This completes the proof.
\qed

Next, we prove a variant of Proposition \ref{mucken2} for $\Phi=\Psi$ with $\tilde{\Phi}\in\Delta_2$. The dual version of the following lemma is established in \cite[Theorem 1]{Lai1993}.
 \begin{lem}\label{mucken1}
 Let $b\in (0,\infty]$ and $\Phi$ be a Young function such that $\Phi,\tilde{\Phi}\in\Delta_2$. Let $w,v$  be locally integrable  functions on $(0,b)$  with $v>0,w>0$ a.e. on $(0,b)$.
Then 
\begin{equation}\label{eq4}
    \int_0^{b}\Phi\left(\left|\int_t^bf(s) ds\right|\right)w(t)dt \leq B_1  \int_0^{b}\Phi(|f(t)|)v(t)dt
\end{equation}
holds for all measurable function $f$ on $(0,b)$ if and only if \begin{equation}\label{eq5} 
 \left( \int_{0}^{t}\epsilon w(s) ds \right)   \varphi\left(\int_{t}^{b}\tilde{\varphi}\left(\frac{1}{\epsilon v(s)}\right)ds\right)\leq B_2,\;\;\;\forall\;\epsilon>0,\,\forall\,t\in (0,b).
 \end{equation}
 Furthermore, for the best constants $B_1$ and $B_2$, there exist positive constants $\alpha_1,\,\alpha_2$, and $C$ depending only on $\Phi$ such that $B_1\leq C\max\{B_2^{\alpha_1},B_2^{\alpha_2}\}$.
 \begin{proof}
First, we derive some inequalities required to prove this lemma. Let $\beta:(0,\infty)\times (0,\infty)\rightarrow (0,\infty)$ such that 
\begin{equation}\label{eqn-beta}
    \beta(r,s)=\frac{1}{\tilde{\Phi}\left(\frac{1}{r s}\right) r s}.
\end{equation}
Using \eqref{eqn-1} (for 
$\tilde{\Phi}$), we get  $\tilde{\Phi}(t)\asymp t\tilde{\varphi}(t)$.
This gives
\begin{equation}\label{eq10}
  \frac{1}{\beta (r,s)}\asymp \tilde{\varphi}\left( \frac{1}{r  s}\right).
 \end{equation}
 Furthermore, using \eqref{eqn-11}  we have
\begin{align*}
   \frac{1}{\beta(r,s)  r s}\leq  \tilde{\Phi}^{-1}\left(\frac{1}{\beta(r,s)  r s}\right)\Phi^{-1}\left(\frac{1}{\beta(r,s) r s}\right)\leq  \frac{2}{\beta(r,s) r s}.
\end{align*}
Now by noticing $ \tilde{\Phi}^{-1}\left(\frac{1}{\beta(r,s) r s}\right)=\frac{1}{r s }$, we obtain 
\begin{equation}\label{eqn::9}
     \frac{1}{\beta(r,s)}\leq  \Phi^{-1}\left(\frac{1}{\beta(r,s)r s}\right)\leq  \frac{2}{\beta(r,s)}.
\end{equation}
 Let $\Phi$, $w$, and $v$ be as given above and denote $z(t)=\int_0^t w(s)ds.$  By taking $\Phi=\Psi$ in  Proposition \ref{mucken2}, we see that \eqref{eq4} holds if and only if there exists  $C=C(\Phi)>0$ such that
 \begin{equation}\label{eq6}
     \frac{1}{\epsilon}\left\|\frac{1}{v}\right\|_{ L^{\tilde{\Phi},\epsilon v}((t,b))}\leq C\Phi^{-1}\left(\frac{1}{\epsilon z(t)}\right),\;\;\; \forall \,\epsilon>0,\,\forall\,t\in (0,b).
 \end{equation}
Our proof will be complete if we show that \eqref{eq5} and \eqref{eq6} are equivalent. First, we 
assume that \eqref{eq5} holds.  
For  $t\in (0,b)$ and $r>0$, using the definition of $\tilde{\varphi}$ and \eqref{eqn-delta2-2-2} (for $\tilde{\varphi}$), we get 
 \begin{equation}\label{eq7}
     \int_t^{b}\tilde{\varphi}\left(\frac{1}{r v(s)}\right)ds\leq \tilde{\varphi}\left(\frac{B_2}{r z(t)}\right)\leq C\max\{B_2^l,B_2^m\} \tilde{\varphi}\left(\frac{1}{r z(t)}\right)
 \end{equation} 
where $l=p_{\tilde{\Phi}}-1,$ $m=q_{\tilde{\Phi}}-1$. 
Combining \eqref{eq10} and \eqref{eq7}, we obtain
 \begin{align*}
   \beta(r,z(t)) \int_t^{b}\tilde{\varphi}\left(\frac{1}{r v(s)}\right)ds\leq C_1,
 \end{align*} 
for some $C_1>0.$ Now using $\tilde{\Phi}(t)\leq t\tilde{\varphi}(t)$ (see \eqref{eqn-9}), we get
\begin{equation*}
    \beta(r,z(t))r \int_t^{b} \tilde{\Phi}\left(\frac{1}{r v(s)}\right) v(s)\, ds\leq \beta(r,z(t))\int_t^{b} \tilde{\varphi}\left(\frac{1}{r v(s)}\right) ds \leq  C_1.
\end{equation*} 
Thus, by the definition of the Luxemburg norm, we get a $C_2>0$ and then using \eqref{eqn::9}, we obtain $$\left\|\frac{1}{r v}\right\|_{L^{\tilde{\Phi}, \beta(r,z(t)) r v}((t,b))}\leq C_2\leq C_2 \beta(r,z(t))\Phi^{-1}\left(\frac{1}{ \beta(r,z(t)) r z(t)}\right).$$
 Notice $\beta(r,z(t))r $ is a continuous function of $r$ and takes all the values in $(0,\infty)$. Thus, for any given $\epsilon>0$, we can choose $r$ such that $\beta(r,z(t)) r=\epsilon$. This concludes \eqref{eq6}.
 
\noindent Conversely, assume that \eqref{eq6} holds.  For  $t\in (0,b)$ and $r>0$  replacing $\epsilon$ by $r\beta(r,z(t))$ in \eqref{eq6} and using \eqref{eqn::9}, we get  $$\frac{1}{r\beta(r,z(t))}\left\|\frac{1}{v} \right\|_{ L^{\tilde{\Phi},r\beta(r,z(t)) v}((t,b))}\leq  C\Phi^{-1}\left(\frac{1}{r\beta(r,z(t)) z(t)}\right)\leq \frac{2C}{\beta(r,z(t))}.$$  Therefore, the definition of the Luxemburg norm gives $$\int_t^{b} \tilde{\Phi}\left(\frac{1}{2Cr v(s)}\right)r\beta(r,z(t)) v(s) ds\leq 1.$$ 
Hence, by using \eqref{eqn-delta2-2}, we conclude that
 \begin{align}
     \int_t^{b} \tilde{\Phi}\left(\frac{1}{r v(s)}\right) r v(s) ds\leq C_3\int_t^{b} \tilde{\Phi}\left(\frac{1}{2Cr v(s)}\right) r v(s) ds\leq \frac{C_3}{\beta(r,z(t))},\label{asbefore}
 \end{align}
 where $C_3=\max\left\{(2C)^{p^-_{\tilde{\Phi}}},(2C)^{p^+_{\tilde{\Phi}}}\right\}.$
Now use $\tilde{\Phi}(t)\asymp t\tilde{\varphi}(t)$, \eqref{asbefore}, and  \eqref{eq10}  to get
\begin{align*}
   \int_t^{b}\tilde{\varphi}\left(\frac{1}{r v(s)}\right)ds\leq C_4\int_t^{b}\tilde{\Phi}\left(\frac{1}{r v(s)}\right)r v(s)ds\leq \frac{C_3C_4}{\beta(r,z(t))}\leq C_5\tilde{\varphi}\left( \frac{1}{r  z(t)}\right)
\end{align*}
for some $C_4,C_5>0.$ Finally, we apply $\varphi$ on both sides of the above inequality and use \eqref{eqn-delta2-2-2}  to obtain the required inequality \eqref{eq5}. The relation between the best constants $B_1, B_2$ can be deduced from the relation between the best constants in Proposition \ref{mucken2}.
 \end{proof}
 \end{lem}

\noindent\textbf{Proof of Theorem \ref{cor}:}
 First, we  show that \eqref{eq5} is satisfied for $w(s)=g^*(s)$ and $v(s)=1/\Phi\left(\zeta(s)\right)$ with $\zeta(s)=s^{(1-N)/N}$. 
 Recall that $ 
Q_\Phi(s)=\Phi(\zeta(s))\tilde{\Phi}\left(\frac{1}{\Phi(\zeta(s))}\right).
 $
Using \eqref{eqn-9} for $\tilde{\Phi}$ we get
 \begin{equation}\label{eq2-thm5}
     Q_\Phi(s)\leq \Phi(\zeta(s))\frac{1}{\Phi(\zeta(s))}\tilde{\varphi}\left(\frac{1}{\Phi(\zeta(s))}\right)= \tilde{\varphi}\left(\frac{1}{\Phi\left(\zeta(s)\right)}\right).
 \end{equation}
 From  \eqref{eqn-7*} using \eqref{eqn-1} (for $\tilde{\Phi}$) and \eqref{eqn-delta2-2-2} (for $\tilde{\varphi}$), we get a constant $C\geq 1$ such that
 $$\tilde{\varphi}(s)\tilde{\varphi}(t)\leq \tilde{\varphi}(Cst),$$ for all $s,t\geq 0.$   Therefore, 
\begin{equation*}
\tilde{\varphi}\left(\frac{\Phi(\zeta(s))}{\epsilon}\right)\tilde{\varphi}\left(\frac{1}{\Phi\left(\zeta(s)\right)}\right)\leq \tilde{\varphi}\left(\frac{C}{\epsilon}\right),\;\;\;\forall\,\epsilon>0.  
\end{equation*}
Now we use \eqref{eq2-thm5} to obtain \begin{equation}\label{eq1-thm5}
 \tilde{\varphi}\left(\frac{1}{\epsilon v(s)}\right)=\tilde{\varphi}\left(\frac{\Phi(\zeta(s))}{\epsilon}\right)\leq \tilde{\varphi}\left(\frac{C}{\epsilon}\right)\frac{1}{\tilde{\varphi}\left(1/\Phi\left(\zeta(s)\right)\right)}\leq \tilde{\varphi}\left(\frac{C}{\epsilon}\right)\frac{1}{Q_\Phi(s)}.   \end{equation}
 Moreover,
 \begin{equation*}
   \int_0^tw(s)ds=\int_0^tg^*(s)ds=tg^{**}(t).
 \end{equation*}
Thus, for $t\in (0,|\Omega|)$ and $\epsilon >0$, we use \eqref{eq1-thm5} and  \eqref{eqn-var} to get
\begin{align*}
    \left( \int_{0}^{t}\epsilon w(s) ds \right) \varphi\left(\int_{t}^{|\Omega|}\tilde{\varphi}\left(\frac{1}{\epsilon v(s)}\right)ds\right)&\leq  \epsilon  tg^{**}(t)\varphi\left(\tilde{\varphi}\left(\frac{C}{\epsilon }\right)\int_{t}^{|\Omega|}\frac{1}{Q_\Phi(s)}ds\right)\\&\leq CC_1t g^{**}(t) \varphi\left(\int_t^{|\Omega|}\frac{1}{Q_\Phi(s)}ds\right)\leq CC_1\|g\|_{X_\Phi(\Omega)},
\end{align*} 
for some $C_1>0.$
 Therefore, by Lemma \ref{mucken1}, there exist positive constants $\alpha_1,\,\alpha_2,$ and $M$  such that inequality \eqref{eq4} holds with $B_1=M\max\left\{\|g\|_{X_\Phi(\Omega)}^{\alpha_1},\|g\|_{X_\Phi(\Omega)}^{\alpha_2}\right\}$. Now we take  $f=-\frac{du^*}{ds}$  in \eqref{eq4} to obtain 
\begin{align}\label{eqn-hl2}
      \int_0^{|\Omega|}g^*(s)\Phi(|u^*(s)|) ds 
   \leq B_1\int_0^{|\Omega|}\frac{1}{\Phi\left(s^{-1+1/N}\right)}\Phi\left(-\frac{du^*}{ds}\right)ds,\;\;\;\forall\,u\in \mathcal{C}^1_c(\Omega).
\end{align}
 Furthermore, since $\Phi\in\Delta^\prime$, there exists $C_2\geq 1$ such that (see \eqref{deld1}) 
 \begin{align*}
   \frac{1}{\Phi\left(s^{-1+1/N}\right)}\Phi\left(-\frac{du^*}{ds}\right)&\leq C_2\Phi\left(N\omega_N^\frac{1}{N}s^{1-1/N}\left(-\frac{du^*}{ds}\right)\right),\;\;\;u\in \mathcal{C}^1_c(\Omega),\;s\in (0,|\Omega|).
\end{align*} 
Consequently, by Proposition \ref{prop1} and \eqref{eqn-hl2}, we conclude
\begin{align*}
    \int_{\Omega}|g(x)|\,\Phi(|u(x)|)dx&\leq C_2M\max\left\{\|g\|_{X_\Phi(\Omega)}^{\alpha_1},\|g\|_{X_\Phi(\Omega)}^{\alpha_2}\right\}\int_{\Omega} \Phi(|\nabla u(x)|)dx,\;\;\; \forall\,u\in \mathcal{C}^1_c(\Omega).
    \end{align*}
Now, \eqref{result-1.6} follows from the above inequality replacing $g$ by $\frac{g}{\|g\|_{X_\Phi(\Omega)}}$. 
\qed
\section{Admissible function spaces in $\mathcal{H}_{\Phi,\Psi}(\Omega).$}
 In this section, we prove Theorem \ref{thm-sim}, Theorem \ref{thm-bdd},  Theorem \ref{theorem3}, and Theorem \ref{capacity}. 

\noindent\textbf{Proof of Theorem \ref{thm-sim}:} Since $\Phi$ and $g$ satisfy  \eqref{modulat phi}, the definition of the Luxemburg norm gives
$$\|u\|_{L^{\Phi,|g|}(\Omega)}\leq C\|\nabla u\|_{L^{\Phi}(\Omega)},\;\;\;\forall\,u\in\mathcal{C}^1_c(\Omega),$$ for some $C>0.$
By Theorem \ref{emb-ele}, there exists a constant $C_1>0$ such that $$\|u\|_{L^{\Psi,|g|}(\Omega)}\leq C_1\|u\|_{L^{\Phi,|g|}(\Omega)},\;\;\;\forall\,u\in\mathcal{C}^1_c(\Omega).$$
Therefore, the above two inequalities yield
$\|u\|_{L^{\Psi,|g|}(\Omega)}\leq CC_1\|\nabla u\|_{L^{\Phi}(\Omega)}$ for all $u\in\mathcal{C}^1_c(\Omega)$. Now the proof follows from Lemma \ref{rem8*}.
\qed

\noindent\textbf{Proof of Theorem \ref{thm-bdd}:}  Let $\Phi,$ $\Psi,$ and $g$ be as given in Theorem \ref{thm-bdd}. Applying Theorem \ref{thm-cianchi}, we get
\begin{equation*}
\|u\|_{L^\infty(\Omega)}\leq C\|\nabla u\|_{L^\Phi(\Omega)},\;\,\,\forall \,u\in \mathcal{C}^1_c(\Omega),
\end{equation*}
for some $C>0.$ Using \eqref{eqn-7*} (replacing $\Phi$ by  $\tilde{\Phi}$) we get a   constant $C_1\geq 1$ such that 
\begin{align*}
  \Phi\left(\|u\|_{L^\infty(\Omega)}\right) \Phi\left(\frac{|\nabla u(x)|}{\|\nabla u\|_{L^\Phi(\Omega)}}\right)\leq C_1 \Phi\left(|\nabla u(x)|\frac{\|u\|_{L^\infty(\Omega)}}{\|\nabla u\|_{L^\Phi(\Omega)}}\right)&\leq C_1\Phi(C|\nabla u(x)|),\;\;\;\forall\,x\in \Omega.
\end{align*}
Integrating the above inequality over $\Omega$, and using $(iii)$ of Proposition \ref{prop2-pell} and \eqref{eqn-delta2-2}, we get 
\begin{align*}
\Phi(\|u\|_{L^\infty(\Omega)})=\Phi\left(\|u\|_{L^\infty(\Omega)}\right)\int_{\Omega}\Phi\left(\frac{|\nabla u(x)|}{\|\nabla u\|_{L^\Phi(\Omega)}}\right)dx&\leq C_2\int_{\Omega}\Phi(|\nabla u(x)|)dx,
\end{align*} 
for some $C_2>0.$
Apply $\Phi^{-1}$ on both sides of the above inequality and use \eqref{eqn----3} to get 
\begin{equation}\label{eqn1-thm7}
    \|u\|_{L^\infty(\Omega)}\leq \max\left\{C_2^{1/p^-_\Phi},C_2^{1/p^+_\Phi}\right\}C_3\Phi^{-1}\left(\int_{\Omega}\Phi(|\nabla u(x)|)dx\right).
\end{equation}
Moreover, 
$$\int_{\Omega}|g(x)|\,\Psi(|u(x)|)dx\leq \Psi\left(\|u\|_{L^\infty(\Omega)}\right)\int_{\Omega}|g(x)|dx=\|g\|_{L^1(\Omega)}\Psi\left(\|u\|_{L^\infty(\Omega)}\right).$$
Now we apply $\Psi^{-1}$ and use \eqref{eqn----3} for $\Psi$ to obtain
\begin{equation}\label{eqn-thm7}
\Psi^{-1}\left(\int_{\Omega}|g(x)|\,\Psi(|u(x)|)dx\right)\leq \max\left\{\|g\|^{1/p^-_\Psi}_{L^1(\Omega)},\|g\|^{1/p^+_\Psi}_{L^1(\Omega)}\right\}\|u\|_{L^\infty(\Omega)}.
\end{equation}
Hence, the conclusion follows from \eqref{eqn1-thm7} and \eqref{eqn-thm7}.
 \qed

\noindent\textbf{Proof of Theorem \ref{theorem3}:}
We show that \eqref{eq19} is satisfied for $w(s)=g^*(s)$ and $v(s)=1/\Phi\left(\zeta(s)\right)$ with $\zeta(s)=s^{(1-N)/N}$.
For  $\epsilon>0$, choose $\delta>0$ such that $\tilde{\Phi}\left(\frac{1}{\delta}\right)=\frac{1}{\epsilon}$. Since $\tilde{\Phi}\in \Delta^\prime$, by \eqref{deld1} there exists $C\geq 1$ such that  for all $s,t>0,$
\begin{equation}\label{asbefore-2}
    \tilde{\Phi}\left(\frac{1}{t  v(s)}\right)\leq C\tilde{\Phi}\left(\frac{1}{\delta}\right)\tilde{\Phi}\left(\frac{\delta}{t v(s)}\right)=\frac{C}{\epsilon} \tilde{\Phi}\left(\frac{\delta}{t v(s)}\right).
\end{equation}
For $r\in (0,|\Omega|),$ we take  $t=\left\|\delta / v\right\|_{L^{\tilde{\Phi},Cv}((r, |\Omega|))}$ in \eqref{asbefore-2} and use $(iii)$ of Proposition \ref{prop2-pell} to get 
 $$\int_r^{|\Omega|}\tilde{\Phi}\left(\frac{1}{t  v(s)}\right)\epsilon v(s) ds \leq  \int_r^{|\Omega|}  \tilde{\Phi}\left(\frac{\delta}{t v(s)}\right)Cv(s) ds=1.$$ 
Therefore, the definition of the Luxemburg norm gives
 \begin{equation*}
   \left\|1/ v\right\|_{L^{\tilde{\Phi}, \epsilon v}((r,|\Omega|))}\leq t=\delta\left\|1 / v\right\|_{L^{\tilde{\Phi},Cv}((r, |\Omega|))}=\frac{1}{\tilde{\Phi}^{-1}(1/\epsilon)}\left\|1/ v\right\|_{L^{\tilde{\Phi},Cv}((r,|\Omega|))}.  
 \end{equation*}   
Multiply both sides of the above inequality by $1/\epsilon$ and use \eqref{eqn-11} to get
 \begin{align}\label{eqn::6.1}
    \frac{1}{\epsilon}  \left\|1/ v\right\|_{L^{\tilde{\Phi}, \epsilon v}((r,|\Omega|))}
    \leq \Phi^{-1}(1/\epsilon)\left\|1/ v\right\|_{L^{\tilde{\Phi},Cv}((r,|\Omega|))}.
 \end{align}
Next, we prove the following estimate:
 \begin{equation}\label{eqn-thm1.9}
     \left\|1/ v\right\|_{L^{\tilde{\Phi},C v}((r,|\Omega|))}
 \leq C^2 \left\|\zeta\right\|_{L^{\tilde{\Phi}}((r,|\Omega|))}.
 \end{equation} 
 Since $\tilde{\Phi}\in \Delta^\prime$, for the same constant $C$ as used in \eqref{asbefore-2}, we have
\begin{align*}
   \tilde{\Phi}\left(\frac{1}{\alpha v(s)}\right)= \tilde{\Phi}\left(\frac{\zeta(s)}{\alpha}\cdot\frac{\Phi(\zeta(s))}{\zeta(s)}\right)&\leq C\tilde{\Phi}\left(\frac{\zeta(s)}{\alpha}\right)\tilde{\Phi}\left(\frac{\Phi(\zeta(s))}{\zeta(s)}\right)\\&\leq C\tilde{\Phi}\left(\frac{\zeta(s)}{\alpha}\right)\Phi\left(\zeta(s)\right),\;\;\;\forall\,\alpha,s\in (0,\infty),
 \end{align*}
 where the last inequality follows from \eqref{eqn-10}.
 Multiply the above inequality by $Cv(s)$ to get  
 \begin{equation*}
     \tilde{\Phi}\left(\frac{1}{\alpha v(s)}\right)Cv(s)\leq C^2\tilde{\Phi}\left(\frac{\zeta(s)}{\alpha}\right)\Phi\left(\zeta(s)\right)v(s)=C^2\tilde{\Phi}\left(\frac{\zeta(s)}{\alpha}\right),\;\;\;\forall\,\alpha,s\in (0,\infty).
 \end{equation*}
Integrate both sides of the above inequality over $(r,|\Omega|)$  to obtain
\begin{align*}
    \int_r^{|\Omega|} \tilde{\Phi}\left(\frac{1}{\alpha v(s)}\right)Cv(s)ds&\leq C^2 \int_r^{|\Omega|}\tilde{\Phi}\left(\frac{\zeta(s)}{\alpha}\right)ds\leq \int_r^{|\Omega|}\tilde{\Phi}\left(\frac{C^2\zeta(s)}{\alpha}\right)ds,\;\;\;\forall \,\alpha>0.
\end{align*}
Hence, \eqref{eqn-thm1.9} follows from the definition of the Luxemburg norm.
Now using \eqref{eqn::6.1}, \eqref{eqn-thm1.9}, and  $\Psi$ is in $\Delta'$, we obtain $C_1\ge 1$ so that 
 \begin{align*}
 \Psi\left(\frac{1}{\epsilon}\left\|1/ v\right\|_{L^{\tilde{\Phi},\epsilon v}((r,|\Omega|))}\right)&\le \Psi\left(C^2\Phi^{-1}(1/\epsilon) \left\|\zeta\right\|_{L^{\tilde{\Phi}}((r,|\Omega|))}\right) \nonumber\\ &\leq C_1 \Psi\circ\Phi^{-1}(1/\epsilon)\,\Psi\left(\left\|\zeta\right\|_{L^{\tilde{\Phi}}((r,|\Omega|))}\right)
 \end{align*}
This gives
 \begin{align*}
    \Psi\left(\frac{1}{\epsilon}\left\|1/ v\right\|_{L^{\tilde{\Phi},\epsilon v}((r,|\Omega|))}\right) \int_0^rw(s)ds
    &\leq C_1\Psi\circ\Phi^{-1}(1/\epsilon)\,\,\Psi\left(\left\|\zeta\right\|_{L^{\tilde{\Phi}}((r,|\Omega|))}\right)rg^{**}(r)\\&\leq C_1\Psi\circ\Phi^{-1}(1/\epsilon)\,\,\|g\|_{X_{\Phi,\Psi}(\Omega)},\;\;\;\,\forall\,r\in (0,|\Omega|). 
\end{align*}
Thus, for all $ r\in (0,|\Omega|)$ and for all $\epsilon>0$,  using \eqref{eqn----3},   we get 
\begin{align*}
\Psi^{-1}\left(\Psi\left(\frac{1}{\epsilon}\left\|1/ v\right\|_{L^{\tilde{\Phi},\epsilon v}((r,|\Omega|))}\right) \int_0^rw(s)ds\right)&\leq  C_1^{1/p_\Psi^-}\max\left\{\|g\|^{1/p^-_\Psi}_{X_{\Phi,\Psi}(\Omega)},\|g\|^{1/p^+_\Psi}_{X_{\Phi,\Psi}(\Omega)}\right\}\Phi^{-1}\left(\frac{1}{\epsilon}\right).
\end{align*}
Therefore, \eqref{eq19} holds and  hence, by  Proposition \ref{mucken2}, there exist positive constants  $\alpha_1,\alpha_2$, and $C_2$ so that  \eqref{eq18} holds. In particular, for $u\in \mathcal{C}_c^1(\Omega),$ by taking $f=-\frac{du^*}{ds}$ in \eqref{eq18}, we obtain
 \begin{equation}\label{eqq2}
     \Psi^{-1}\left(\int_0^{|\Omega|}g^*(s)\Psi(|u^*(s)|) ds\right) \leq B\Phi^{-1}\left( \int_0^{|\Omega|}\frac{1}{\Phi\left(s^{-1+1/N}\right)}\Phi\left(-\frac{du^*}{ds}\right)ds\right),
 \end{equation}
 where $B=C_2\max\left\{\|g\|_{X_{\Phi,\Psi}(\Omega)}^{\alpha_1},\|g\|_{X_{\Phi,\Psi}(\Omega)}^{\alpha_2}\right\}$.
Moreover, by \eqref{deld1}  there exists $C_3\geq 1$ such that 
 \begin{align*}
   \frac{1}{\Phi\left(s^{-1+1/N}\right)}\Phi\left(-\frac{du^*}{ds}\right)&\leq C_3\Phi\left(N\omega_N^\frac{1}{N}s^{1-1/N}\left(-\frac{du^*}{ds}\right)\right),\;\;\;u\in \mathcal{C}^1_c(\Omega),\;s\in (0,|\Omega|).
\end{align*} 
Consequently, by  Proposition \ref{prop1} (Hardy-Littlewood and P\'olya-Szeg\"o) and \eqref{eqq2} we get
 \begin{align*}
\Psi^{-1}\left(\int_{\Omega}|g(x)|\,\Psi(|u(x)|)dx\right) &
    \leq BC_3 \Phi^{-1}\left(\int_{\Omega}\Phi(|\nabla u(x)|)dx\right),\;\;\;\forall\, u\in \mathcal{C}_c^1(\Omega).
 \end{align*}
Next we replace $g$ by $g/\|g\|_{X_{\Phi,\Psi}(\Omega)}$ in the above inequality to obtain
 \begin{align*}
\Psi^{-1}\left(\frac{1}{\|g\|_{X_{\Phi,\Psi}(\Omega)}}\int_{\Omega}|g(x)|\,\Psi(|u(x)|)dx\right) &
    \leq C_2C_3\Phi^{-1}\left(\int_{\Omega}\Phi(|\nabla u(x)|)dx\right),\;\;\;\forall\,u\in \mathcal{C}_c^1(\Omega).
 \end{align*}
Now using \eqref{eqn----3} we get \eqref{eqn-thm3}.
 This completes the proof.
\qed

\noindent\textbf{Proof of Theorem \ref{capacity}:}
First, assume that condition $\textit{(i)}$ holds. Thus, 
\begin{equation*}
    \Psi^{-1}\left(\int_{\Omega}|g(x)|\,\Psi(|u(x)|)dx \right)\leq C\Phi^{-1}\left(\int_{\Omega}\Phi(|\nabla u(x)|) dx \right),\,\,\,\forall\,u\in \mathcal{C}^1_c(\Omega).
\end{equation*}
Let $K$ be a compact set and $u\in \mathcal{C}^1_c(\Omega)$ be such that $u(x)\geq 1$, $x\in K$. Thus, by \eqref{eqn-delta2-2}  we get
\begin{align*}
    \Psi(1)\int_K|g(x)|dx\leq \int_{\Omega}|g(x)|\Psi(|u(x)|) dx\leq \max\left\{C^{p_\Psi^-},C^{p^+_\Psi}\right\}\Psi\circ\Phi^{-1}\left(  \int_{\Omega}\Phi(|\nabla u(x)|) dx  \right).
\end{align*}
By taking infimum over all such $u$, we obtain
 $$\Psi(1)\int_K|g(x)|dx\leq \max\left\{C^{p_\Psi^-},C^{p^+_\Psi}\right\}\Psi\circ\Phi^{-1}\left(\text{Cap}_{\Phi}(K,\Omega)  \right).$$ 
  Therefore, $\textit{(ii)}$ holds and the best constant $D$ in $\textit{(ii)}$ satisfies $D\leq \max\left\{C^{p_\Psi^-},C^{p^+_\Psi}\right\}/\Psi(1)$.
 
\noindent Conversely, assume that condition $\textit{(ii)}$ holds. For $u\in \mathcal{C}^1_c(\Omega)$ and $k\in \mathbb{Z}$ denote $$E_k=\{x\in \Omega :|u(x)|>2^k\}\;\; \text{and} \; \; A_k=E_k\setminus E_{k+1}.$$
 Observe that $$\Omega=\{x\in \Omega:0\leq |u(x)|<\infty\}=\{x\in \Omega :u(x)=0\}\cup \bigcup_{i\in \mathbb{Z}}A_i.$$
Now using condition $\textit{(ii)},$ we get 
 \begin{align}
     \int_{\Omega}|g(x)|\Psi(|u(x)|)dx\nonumber &= \sum_{k\in \mathbb{Z}}\int_{A_{k+1}}|g(x)|\Psi(|u(x)|)dx\nonumber \\&\leq\sum_{k\in \mathbb{Z}}\Psi(2^{k+2})\int_{A_{k+1}}|g(x)|dx\nonumber\\&\leq D \sum_{k\in \mathbb{Z}} \Psi(2^{k+2})\Psi\circ\Phi^{-1}\left(  \text{Cap}_{\Phi}(\overline{A_{k+1}},\Omega)\right)\nonumber\\&\leq  4^{p^+_\Psi}D\sum_{k\in \mathbb{Z}} \Psi(2^{k})\Psi\circ\Phi^{-1}\left(  \text{Cap}_{\Phi}(\overline{A_{k+1}},\Omega)\right)\label{eqn::7},
\end{align}
where the last inequality follows from \eqref{eqn-delta2-2}. Furthermore, by Proposition \ref{lemcap}, there exists $C_2\geq 1$ such that 
\begin{align}
\Psi(2^{k})\Psi\circ\Phi^{-1}\left(  \text{Cap}_{\Phi}(\overline{A_{k+1}},\Omega)\right)&=\Psi\circ\Phi^{-1}\left(\Phi(2^{k})\right)\Psi\circ\Phi^{-1}\left(  \text{Cap}_{\Phi}(\overline{A_{k+1}},\Omega)\right)\nonumber\\& \leq C_2\Psi\circ\Phi^{-1}\left(\Phi(2^{k}) \text{Cap}_{\Phi}(\overline{A_{k+1}},\Omega)\right).\label{eqn1-cap}
 \end{align}
Next to estimate $\text{Cap}_{\Phi}(\overline{A_{k+1}},\Omega)$, we choose a smooth function $\alpha:[0,1]\rightarrow \mathbb{R}$ satisfying 
 \[   
\alpha(t) = 
     \begin{cases}
       \text{0} &\quad\text{if $0\le t <\frac{1}{4}$},\\
       \text{1} &\quad\text{if $ \frac{1}{2}\leq t\le 1$ }.\\ 
     \end{cases}
\]
For each $k\in \mathbb{Z}$, define $u_k:\Omega\rightarrow [0, 1] $ by  
 \[   
u_k(x) = 
     \begin{cases}
       \text{1} &\quad\text{if $|u(x)|\geq 2^{k+1}$},\\
        \text{$\alpha\left(\frac{|u(x)|}{2^k}-1\right)$} &\quad\text{if $2^k<|u(x)|<2^{k+1}$},\\
        \text{0} &\quad\text{if $|u(x)|\leq 2^k$}.\\
     \end{cases}
\]
Clearly $u_k\in \mathcal{C}_c^1(\Omega)$ and $u_k\equiv 1$ on $\overline{E_{k+1}}\supset \overline{A_{k+1}}$. Moreover, $|\nabla u_k(x)|\leq \|\alpha^\prime\|_{L^\infty([0,1])}\frac{|\nabla u(x)|}{2^k}$, for $x\in \Omega $ such that $2^k<|u(x)|\leq 2^{k+1}$. Therefore, the definition of capacity and \eqref{eqn-delta2-2} gives
\begin{align}
    \text{Cap}_{\Phi}(\overline{A_{k+1}},\Omega)\leq \int_{\Omega}\Phi(|\nabla u_k(x)|)dx&\leq \int_{A_k}\Phi\left(\frac{\|\alpha^{\prime}\|_{L^\infty([0,1])}|\nabla u(x)|}{2^k}\right)dx\nonumber\\&\leq C_1\int_{A_k}\Phi\left(\frac{|\nabla u(x)|}{2^k}\right)dx,\label{eqn2-cap}
\end{align}
where $C_1=\max\left\{\|\alpha^{\prime}\|_{L^\infty([0,1])}^{p^-_\Phi},\|\alpha^{\prime}\|_{L^\infty([0,1])}^{p^+_\Phi}\right\}$. 
Since $\tilde{\Phi}\in \Delta^\prime,$ by \eqref{eqn-7*} we get a constant $C_3\geq 1$  such that
 \begin{align}\label{eqn3-cap}
\Phi(2^{k})\Phi\left(\frac{|\nabla u(x)|}{2^k}\right)\leq  C_3\Phi\left(|\nabla u(x)|\right),\;\;\;x\in \Omega.
 \end{align}
Combining \eqref{eqn1-cap}, \eqref{eqn2-cap}, and \eqref{eqn3-cap}, we obtain
\begin{align}
\Psi(2^{k})\Psi\circ\Phi^{-1}\left(\text{Cap}_{\Phi}(\overline{A_{k+1}},\Omega)\right)&\leq  C_2 \Psi\circ\Phi^{-1}\left(C_1\Phi(2^k)\int_{A_k}\Phi\left(\frac{|\nabla u(x)|}{2^k}\right)dx\right)\nonumber\\&\leq C_2 \Psi\circ\Phi^{-1}\left(C_1C_3\int_{A_k}\Phi\left(|\nabla u(x)|\right)dx\right).\label{eqn-thm5-1}
 \end{align}
Since $\Psi\circ\Phi^{-1}$ is super-additive, there exists $C_4>0$ such that (see Definition \ref{dfn1})
 \begin{align*}
 \sum_{k\in \mathbb{Z}} \Psi\circ\Phi^{-1}\left(C_1C_3\int_{A_k}\Phi\left(|\nabla u(x)|\right)dx\right)&\leq C_4 \Psi\circ\Phi^{-1}\left(C_1C_3\sum_{k\in \mathbb{Z}}\int_{A_k}\Phi\left(|\nabla u(x)|\right)dx\right)\nonumber\\&= C_4 \Psi\circ\Phi^{-1}\left(C_1C_3\int_{\Omega}\Phi\left(|\nabla u(x)|\right)dx\right).
 \end{align*}
Thus, from \eqref{eqn::7} and \eqref{eqn-thm5-1}  we get 
 \begin{align*}
   \int_\Omega|g(x)|\Psi(|u(x)|)dx&\leq  4^{p^+_\Psi}C_2C_4D\Psi\circ\Phi^{-1}\left(C_1C_3\int_{\Omega}\Phi\left(|\nabla u(x)|\right)dx\right).
 \end{align*}
Apply $\Psi^{-1}$ on both sides of  the above inequality and use \eqref{eqn-6} (for $\tilde{\Psi}$) to get
\begin{align*}
    \Psi^{-1}\left(\int_{\Omega}|g(x)|\Psi(|u(x)|)dx\right)&\leq C_5\Psi^{-1}\left( D\right)\Phi^{-1}\left(\int_{\Omega} \Phi\left(|\nabla u(x)|\right)dx\right),
\end{align*}
for some $C_5>0$. 
Hence $\textit{(i)}$ holds, and the best constant $C$ in \eqref{modular} satisfies $C\leq C_5\Psi^{-1}(D)$.
  \qed

\section{Compactness and existence of the solutions.}
In this section, we prove the existence of the eigenvalues of \eqref{eqJ}. Towards this, we prove the following elementary lemma.
\begin{lem}\label{lemma:L1}
Let $\Psi$ be a Young function satisfying the $\Delta_2$-condition, and $(u_n)$ be a sequence such that $u_n\to u$  in $L^\Psi(\Omega)$. Then  $\Psi(|u_n|)\to \Psi(|u|)$  in $L^1(\Omega)$.
\begin{proof}
Since $\psi$ is the right derivatives of $\Psi$, 
$$|\Psi(s)-\Psi(t)|\leq \psi(\max\{s,t\})|s-t|\leq \psi(s+t)|s-t|,\;\;\;\forall\, s,t\in (0,\infty).$$
Now, the H\"older's inequality for Young functions gives
 \begin{align}\label{eqW}
     \int_{\Omega}|\Psi(|u_n(x)|)-\Psi(|u(x)|)|\,dx&\leq \int_{\Omega}\psi(|u_n(x)|+|u(x)|)\;|u_n(x)-u(x)|\,dx\nonumber\\&\leq 2\|\psi(|u_n|+|u|)\|_{L^{\tilde{\Psi}}(\Omega)}\|u_n-u\|_{L^\Psi(\Omega)}.
 \end{align}
 Next, we estimate $\|\psi(|u_n|+|u|)\|_{L^{\tilde{\Psi}}(\Omega)}$. 
By \eqref{eqn-9} 
we get $\tilde{\Psi}(t)\leq t\tilde{\psi}(t)$ and $t\psi(t)\leq \Psi(2t),$ and $\tilde{\psi}\circ\psi(t)\leq t.$ Thus, we obtain
 \begin{align*}
\int_{\Omega}\tilde{\Psi}\left(\psi(|u_n(x)|+|u(x)|)\right)dx& \leq \int_{\Omega}\psi\left(|u_n(x)|+|u(x)|\right)\;\tilde{\psi}\circ\psi\left(|u_n(x)|+|u(x)|\right)\,dx\\&\leq \int_{\Omega}\Psi\left(2(|u_n(x)|+|u(x)|)\right)dx.
\end{align*}
Therefore, the definition of the Luxembourg gives
we have
\begin{equation}\label{eqn-lem-com}
    \|\psi(|u_n|+|u|)\|_{L^{\tilde{\Psi}}(\Omega)}\leq 2\|u_n+u\|_{L^\Psi(\Omega)}.
\end{equation} 
Since $u_n\to u$ in $L^\Psi(\Omega)$, there exists $B_1>0$ such that  $\|u_n+u\|_{L^\Psi(\Omega)} \leq B_1$, for every $n\in\mathbb{N}$. Hence, the result follows from  \eqref{eqW} and \eqref{eqn-lem-com}.
\end{proof}
\end{lem}
\begin{lem}[{\it{Lagrange multipliers theorem}}]\cite[Theorem 4]{Browder1965}\label{lagrange}
Let $V$ be a Banach space, $f$ and $g$ be two real-valued functions on $V$ that are Fréchet differentiable at $v_0\in V.$ If $g^\prime(v_0)\neq 0$ and $v_0$ is a point of local minimum of $f$ with respect to the set $\{v : g(v) = g(v_0)\},$
then there exists $\lambda\in \mathbb{R}$ such that $f^\prime(v_0)=\lambda g^\prime(v_0).$
\end{lem}
Now, we are ready to prove Theorem \ref{solu}.

\noindent\textbf{Proof of Theorem \ref{solu}:} Let $G_\Psi$ be as given in  \eqref{eqn-G-Psi}. Our proof is divided into two steps. 

 {\it{\underline{\textbf{$G_\Psi$ is compact:}}}} We adapt the proof of \cite[Lemma 6.1]{Anoop2021} to our case.  Let $(u_n)$ be a sequence that converges weakly to $u$ in $\mathcal{D}^{1,\Phi}_0(\Omega)$. Then  there exists a constant  $B\in [0,\infty)$ such that $$B=\sup_{n\in \mathbb{N}}\left\{\int_{\Omega}\left(\Phi(|\nabla u|)+\Phi(|\nabla u_n|)\right)dx\right\}.$$ Let $0<\epsilon<1$. Since $g\in \mathcal{F}_V(\Omega)$ and $\delta$ is a positive function satisfying $\delta(t)\to 0$ as $t\to 0$, there exists $g_\epsilon\in \mathcal{C}_c(\Omega)$ such that $\delta\left(\|g-g_\epsilon\|_{V(\Omega)}\right)< \epsilon$. For $K=supp(g_\epsilon)$, observe that

 \begin{align}\label{eqL}
     G_{\Psi}(u_n)-G_{\Psi}(u)&=\int_{\Omega}g\left(\Psi(|u_n|)-\Psi(|u|)\right)dx\nonumber\\&\leq \int_{K}|g_\epsilon|\left|\Psi(|u_n|)-\Psi(|u|)\right| dx+\int_{\Omega}(|g-g_\epsilon|)\left|\Psi(|u_n|)-\Psi(|u|)\right| dx. 
 \end{align}
 We estimate the second integral on the right-hand side of the above inequality, using \eqref{eqK} as
  \begin{align*}
  \int_{\Omega}(|g-g_\epsilon|)\left|\Psi(|u_n|)-\Psi(|u|)\right| dx&\leq  \int_{\Omega}(|g-g_\epsilon|)\Psi(|u_n|) dx+ \int_{\Omega}(|g-g_\epsilon|)\Psi(|u|)dx\\&\leq \epsilon \Psi\circ\Phi^{-1}\left(\int_{\Omega}\Phi(|\nabla u_n|)dx\right)+\epsilon \Psi\circ\Phi^{-1}\left(\int_{\Omega}\Phi(|\nabla u|)dx\right)\\&\leq 2\epsilon \Psi\circ\Phi^{-1}(B).
 \end{align*}
  Moreover, by Proposition \ref{prop:comp} and Lemma \ref{lemma:L1}, there exists $n_1\in \mathbb{N}$ such that 
 $$\int_{K}|g_\epsilon|\left|\Psi(|u_n|)-\Psi(|u|)\right|dx<\epsilon, \;\;\;\forall\, n\geq n_1. $$
Thus, from (\ref{eqL}) we obtain $$ |G_{\Psi}(u_n)-G_{\Psi}(u)|<\left(2\Psi\circ\Phi^{-1}(B)+1\right)\epsilon,\;\;\; \forall\, n\geq n_1.$$
 Hence $G_{\Psi}(u_n)\rightarrow G_{\Psi}(u)$  and consequently $G_\Psi$ is compact on $\mathcal{D}^{1,\Phi}_0(\Omega)$.

{\it{\underline{\textbf{Existence of solutions:}}}} By $(ii)$ of Proposition \ref{prop2-pell}, we have $$J_\Phi(u)=\int_{\Omega}\Phi(|\nabla u|) dx\geq \|\nabla u\|_{L^\Phi(\Omega)}-1.$$ 
Thus, $J_\Phi$  is coercive. Since $\Phi,\tilde{\Phi}\in\Delta_2$,
$\mathcal{D}^{1,\Phi}_0(\Omega)$ is a reflexive space (see \cite[Proposition 3.1]{Cianchi2017}). Recall that $\lambda_1(r)=\inf\left\{J_\Phi(u):u\in N_r\right\}$ and $ N_r=\left\{u\in \mathcal{D}^{1,\Phi}_0(\Omega):G_\Psi(u)=r\right\}$.  Let $(u_n)$ be a minimizing sequence for $\lambda_1(r)$ on the set $N_r$.  By coercivity of $J_\Phi$ and reflexivity of  $\mathcal{D}^{1,\Phi}_0(\Omega)$, there exists a
sub-sequence $(u_{n_k})$  converging weakly to $u_1$ in
$\mathcal{D}^{1,\Phi}_0(\Omega)$. Now using the compactness of $G_{\Psi}$, we have $u_1\in N_r$. Furthermore, the lower semi-continuity of $J_\Phi$ with respect to the
weak convergence (see \cite[Theorem 2.2.8]{Lars2011}) gives $$ \lambda_1(r)=\underline{\lim}_{k\to \infty}\int_{\Omega}\Phi(|\nabla u_{n_k}|)dx\geq \int_{\Omega}\Phi(|\nabla u_1|)dx\geq  \lambda_1(r). $$
 Therefore, $\lambda_1(r)$ is attained, and $J_\Phi$ admits a minimizer $u_1$ over $N_r$. 
Moreover, $J_\Phi$ and $G_\Psi$ are Fréchet derivable (see \cite[Proposition 2.17]{salort2021}, \cite[Lemma A.3]{Fukagai2006}) 
  with  derivatives  given by $$\langle J_\Phi^{\prime}(u), v \rangle=\int_{\Omega}\varphi(|\nabla u|)\frac{\nabla u}{|\nabla u|}\cdot\nabla v\, dx, \;\;\;\,\langle G_\Psi^{\prime}(u), v\rangle=\int_{\Omega}g\psi(|u|)\frac{uv}{|u|}dx.$$
Now using $t\psi(t)\asymp \Psi(t)$ (see \eqref{eqn-1}) we get $$\langle G_\Psi^{\prime}(u_1), u_1\rangle=\int_{\Omega}g\psi(|u_1|)|u_1|dx\asymp \int_{\Omega}g\Psi(|u_1|)dx=G_\Psi(u_1)=r\neq 0.$$
Thus, $G^\prime_\Psi(u_1)\neq 0$ and hence by 
Lemma \ref{lagrange}, there exists $\tilde{\lambda}_1(r)\in \mathbb{R}$ such that  $$\int_{\Omega}\varphi(|\nabla u_1|)\frac{\nabla u_1}{|\nabla u_1|}\cdot\nabla v dx= \tilde{\lambda}_1(r)\int_{\Omega}g\psi(|u_1|)\frac{u_1 v}{|u_1|}dx,\;\;\;\forall \,v\in \mathcal{D}^{1,\Phi}_0(\Omega).$$
Since $|u_1|\in  \mathcal{D}^{1,\Phi}_0(\Omega)$, $J_\Phi(|u_1|)=J_\Phi(u_1)$, and $G_\Psi(|u_1|)=G_\Psi(u_1)$, we can take $u_1(x)\geq 0,$ for a.e. $x\in \Omega$. As $g$ is non-zero non-negative, from the above inequality, we obtain $\tilde{\lambda}_1(r)>0.$ 
This completes the proof.
\qed

 \section{Concluding remarks:}
This section provides various examples of Young functions satisfying the assumptions of theorems in the introduction and identifies the associated admissible function spaces for the weight function $g$. We relate the admissible function spaces for $g$ with some classical function spaces such as Lorentz and Lorentz-Zygmund spaces. For $q\in (0,\infty),$ recall the Lorentz and Lorentz-Zygmund spaces 
$$L^{q,\infty}(\Omega)=\left\{g \in \mathcal{M}(\Omega): \sup_{0<t<|\Omega|}t^{1/q}g^{**}(t)<\infty \right\},$$
     $$\mathcal{L}^{1,\infty;q}(\Omega)=\left\{g\in \mathcal{M}(\Omega): \sup_{0<t<|\Omega|}tg^{*}(t)\left(\log\left(\frac{e|\Omega|}{t}\right)\right)^q<\infty\right\}.$$  
\begin{rem}\label{exmp5.1} Consider the Young function $A_p(t):=t^p$ with $p\in(1,\infty).$  Clearly, $\,p_{A_p}^+=p$ and $\tilde{A_p}\asymp t^{p^\prime}$. Depending on $N$ and $p$, we see below that the admissible function spaces for  $g$ given in Theorem \ref{thm:3}, Theorem \ref{larentzthm}, and Theorem \ref{cor}  correspond to certain Lorentz and Lorentz Zygmund spaces.
\begin{enumerate} [(i)] 
     \item For  $\Phi=A_p$ with  $N>p$, both Theorem \ref{thm:3} and  Theorem \ref{larentzthm} are applicable.
    In this case, Theorem \ref{thm:3} gives $L^{\tilde{B}_\Phi}(\Omega)=L^{N/p}(\Omega)$ and  Theorem \ref{larentzthm} gives a larger space
    $L^{\Phi,\infty}(\Omega)=L^{N/p,\infty}(\Omega).$

       \item 
       For $\Phi=A_p,$ next we find the function space $X_\Phi(\Omega)$ given by Theorem \ref{cor}. Recall that 
 $$ X_{\Phi}(\Omega)=\left\{g\in \mathcal{M}(\Omega) :\sup_{0<r<|\Omega|}\left\{g^{**}(r)\eta_\Phi(r)\right\}<\infty\right\}.$$ It is easy to compute that (see \eqref{eqA*} and \eqref{eta-Phi})  $$Q_\Phi(s)\asymp s^\frac{p(N-1)}{N(p-1)},\;\;\;\forall\,s>0,$$ and for $0<r<|\Omega|$, 
 \[\eta_\Phi(r) \asymp
     \begin{cases}
     \text{$r^\frac{p}{N}$} &\quad\text{if $N> p,\,|\Omega|=\infty$},\\
     \text{$r\left(\log\left(\frac{|\Omega|}{r}\right)\right)^{N-1}$} & \quad\text{if $N=p,\,|\Omega|<\infty$},\\ \text{$r\left|r^\frac{p-N}{N(p-1)}-|\Omega|^\frac{p-N}{N(p-1)}\right|^{p-1}$} &\quad\text{if  $N\neq p,\,|\Omega|<\infty$.}
     \end{cases}
\] 
Thus, $\Phi$ satisfies \eqref{8} for all the above three cases. Now, one can identify the following:
\[X_\Phi(\Omega) =
     \begin{cases}
     \text{$L^{\frac{N}{p},\infty}(\Omega)$} &\quad\text{if $N> p,$}\\
     \text{$\mathcal{L}^{1,\infty;N}(\Omega)$} & \quad\text{if $N=p,\,|\Omega|<\infty$},\\ \text{$L^1(\Omega)$} &\quad\text{if  $N< p,\,|\Omega|<\infty$.}
     \end{cases}
\]
Note that, for $N=p$, we need to use argument as in \cite[Proposition A.1] {ujjal2020}. 
 \end{enumerate} 
  \end{rem}

  \begin{rem}\label{exm-1}  Now we consider some Young function $\Phi$ other than $A_p$ and try to find the associated admissible function spaces given by  
  Theorem \ref{larentzthm} and Theorem \ref{cor}.
   \begin{enumerate}[(i)] 
  \item
For $1<p<q,$ let $p\,\overline{q}<N\overline{p}\,$ and
      \[   
\Phi(t) = 
      \begin{cases}
        \text{$\frac{t^{q}}{\overline{q}}$} &\quad\text{if $ t\in [0,1]$},\\
      \text{$\frac{1}{\overline{p}}(t^{p}-1)+\frac{1}{\overline{q}}$}&\quad\text{if $t\in (1,\infty)$},
     \end{cases}
 \] 
where $\overline{q}=q(q^\prime)^{q-1}$ and $\overline{p}=p(p^\prime)^{p-1} $. Then, one can verify that  $p^+_\Phi=\frac{p\,\overline{q}}{\overline{p}}$ and
   \[   
 \tilde{\Phi}(t) = 
       \begin{cases}
         \text{$t^{q^\prime}$} &\quad\text{if $ t\in [0,1]$},\\
       \text{$t^{p^\prime}$}&\quad\text{if $t\in (1,\infty)$}.
      \end{cases}
  \] 
Moreover, $\Phi\in \Delta_2$  and $\tilde{\Phi}\in \Delta^\prime$ (see Example \ref{exam}). Thus, $\Phi$ satisfies all the assumptions of Theorem \ref{larentzthm} and hence $L^{\Phi,\infty}(\Omega)\subset \mathcal{H}_{\Phi,\Phi}(\Omega)$. 
In fact, if $|\Omega|<\infty,$ then one can identify that $L^{\Phi,\infty}(\Omega)=L^{\frac{N}{q},\infty}(\Omega).$ In general, $L^{\frac{N}{p},\infty}(\Omega)\cap L^{\frac{N}{q},\infty}(\Omega) \subset L^{\Phi,\infty}(\Omega)$. 

\item  Notice that the Young function considered in the above example does not satisfy the $\Delta^\prime$-condition and hence Theorem \ref{cor} is not applicable for this Young function. 
For $N>p$ and $p \leq q < \min\left\{\frac{Np-1}{N-1},\frac{N+Np-p}{N}\right\}$, consider  $\Phi(t)=\max\{t^p,t^q\}.$ To identify the admissible function space $X_\Phi(\Omega)$ as given in  Theorem \ref{cor} we need to compute $\eta_\Phi$. Recall that $$\eta_\Phi(r)= r\varphi\left(\int_r^{|\Omega|}\frac{1}{Q_\Phi(s)}ds\right),\;\; r\in (0,|\Omega|) $$
where  $$ Q_\Phi(s) =\Phi(\zeta(s))\tilde{\Phi}\left(\frac{1}{\Phi(\zeta(s))}\right),\; s>0$$ 
with $\zeta(s)=s^{\frac{1}{N}-1}$.
 Since $\tilde{\tilde{\Phi}}=\Phi,$ $(i)$ of Remark \ref{exm-1} gives $\tilde{\Phi}(t)\asymp t^{p^\prime}$ for $t\in (0,1)$ and $\tilde{\Phi}(t)\asymp t^{q^\prime}$ for $t\in [1,\infty)$. 
Now, one can verify that 
\[
Q_\Phi(s) \asymp
     \begin{cases}
       \text{$\left(\frac{1}{{\Phi(\zeta(s))}}\right)^{1/(p-1)}= s^\frac{q(N-1)}{N(p-1)}$} &\quad\text{if $s\in(0, 1)$},\\
     \text{$\left(\frac{1}{{\Phi(\zeta(s))}}\right)^{1/(q-1)}= s^\frac{p(N-1)}{N(q-1)}$}&\quad\text{if $s\in [1,\infty)$}.
     \end{cases}
\]
Observe that  $\varphi(t)=pt^{p-1}$ if $t<1$ and $\varphi(t)=qt^{q-1}$ if $t\geq 1$. 
Thus, for all $r\in (0,|\Omega|)$ we have
\[
\eta_\Phi(r)\asymp 
 \begin{cases}
       \text{$ r^{\alpha(q-1)+1}
       $} &\quad\text{near zero},\\
      \text{$r^{\beta(p-1)+1}$}&\quad\text{near infinity when $|\Omega|=\infty$,}
     \end{cases}
\]
where $\alpha=\frac{N(p-1)-q(N-1)}{N(p-1)} <0 $ and $\beta=\frac{N(q-1)-p(N-1)}{N(q-1)} <0.$
Consequently, 
$\lim_{r\to 0}\eta_\Phi(r)=0$ (as $\alpha(q-1)+1>0$), i.e. \eqref{8} is satisfied. Note that $\Phi$ satisfies all assumptions of Theorem \ref{cor} and hence $X_\Phi(\Omega)\subset \mathcal{H}_{\Phi,\Phi}(\Omega)$. 
Moreover, from the expression of $\eta_\Phi,$ one can identify that   $X_\Phi(\Omega)=L^{\frac{1}{\alpha(q-1)+1},\infty}(\Omega)$ if $|\Omega|<\infty$. In general,     
$L^{\frac{1}{\alpha(q-1)+1},\infty}(\Omega)\cap L^{\frac{1}{\beta(p-1)+1},\infty}(\Omega)\subset X_\Phi(\Omega)$.
\end{enumerate}
\end{rem}

\begin{rem}\label{exam-5.3} Next, considering $\Phi=A_p$ and $\Psi=A_q$ with $p,q\in (1,\infty),$ we identify admissible function spaces for $g$ given by Theorem \ref{thm-bdd} and Theorem \ref{theorem3}. 
        \begin{enumerate}[(i)]
        \item If $N<p$ and $|\Omega|<\infty$, then $\Phi$ satisfies \eqref{eqG1}. Thus, by Theorem \ref{thm-bdd}, we conclude that $L^1(\Omega)\subset \mathcal{H}_{\Phi,\Psi}(\Omega)$.
            \item    
Let $q \geq p$, and $q \leq p^*$ when $N>p$.  Recall  $\eta_{\Phi,\Psi}$ as given in \eqref{eta-phi,psi}. For $r \in (0,|\Omega|)$, one can check that
    \[\eta_{\Phi,\Psi}(r) \asymp
     \begin{cases} 
       \text{$r^\frac{N(p-q)+pq}{Np}$} &\quad \text{if $N> p,\,|\Omega|=\infty$,}\\
       \text{$r\left(\log\left(\frac{|\Omega|}{r}\right)\right)^\frac{q}{N^\prime}$} & \quad\text{if $N=p,\,|\Omega|<\infty$},\\\text{$r\left|r^\frac{p-N}{N(p-1)}-|\Omega|^\frac{p-N}{N(p-1)}\right|^\frac{q(p-1)}{p}$} &\quad   \text{if $N\neq  p,\, |\Omega|<\infty $.}
     \end{cases}
\]
 Notice that $\Psi\circ\Phi^{-1}$ is super-additive (as  $q\geq p$).  For $N\leq p$ and $|\Omega|<\infty$, \eqref{7} holds for any $q>1$. However, for $N>p$,  \eqref{7} holds only if $q\leq p^*$.  Thus, Theorem \ref{theorem3} is applicable in these cases. This concludes $X_{\Phi,\Psi}(\Omega) \subset\mathcal{H}_{\Phi,\Psi}(\Omega)$.  Furthermore, one can identify that 
\[X_{\Phi,\Psi}(\Omega)=
     \begin{cases}
     \text{$L^{\left(\frac{p^*}{q}\right)^\prime,\infty}(\Omega)$} &\quad\text{if $N>p$ and $ q\in [p, p^*],$}\\
     \text{$\mathcal{L}^{1,\infty;\frac{q+N^\prime}{N^\prime}}(\Omega)$} & \quad\text{if $N=p\leq q$ and $|\Omega|<\infty$},\\ \text{$L^1(\Omega)$} &\quad\text{if  $N<p\leq q$ and $|\Omega|<\infty$.}
     \end{cases}
\]
This shows that Theorem \ref{theorem3} extends \cite[Theorem 1.1]{Visciglia2005} and \cite[Theorem 1.2]{Anoop2021}.
\end{enumerate}
\end{rem}

\begin{rem} In this remark, we give examples of Young functions $\Psi$ other than $A_q$ that are applicable for Theorem \ref{thm-bdd} and Theorem \ref{theorem3}.
\begin{enumerate}[(i)]
    \item 
For $q>p>N$ and $|\Omega|<\infty$, consider $\Phi=A_p$ and $\Psi(t)=\max\{t^p,t^q\}.$
 Since $p_\Phi^-=p>N,$ it follows that $\Phi$ satisfies \eqref{eqG1}.
Therefore, by Theorem \ref{thm-bdd}, we have $L^1(\Omega)\subset \mathcal{H}_{\Phi,\Psi}(\Omega)$.
\item  For $p \in (1,N)$ and $q \in[ p,p^*]$, consider $\Phi=A_p$ and $\Psi(t)=\max\{t^p,t^q\}$. 
Then, we can see that $\tilde{\Phi}\asymp A_{p^\prime},$ and $\Psi\circ \Phi^{-1}(t)=\max\{t,t^{q/p}\}$ is a super-additive function (as it is convex).  To determine the admissible function space $X_{\Phi,\Psi}(\Omega)$ given in Theorem \ref{theorem3}, we first understand the behaviour of $\eta_{\Phi,\Psi}$. Recall that
 $$\eta_{\Phi,\Psi}(r)=r\Psi\left(\|\zeta\|_{L^{\tilde{\Phi}}((r,|\Omega|))}\right),\;\;\;\,r\in (0,|\Omega|),$$
where $\zeta(s)=s^{\frac{1}{N}-1}$. Since $\tilde{\Phi}\asymp A_{p^\prime}$ and $N>p$, one can compute that for all $r\in (0,|\Omega|),$
\[
\|\zeta\|_{L^{\tilde{\Phi}}((r,|\Omega|))}\asymp\|\zeta\|_{L^{p^\prime}((r,|\Omega|))}\asymp 
\begin{cases}
\text{$\left (r^{\frac{p-N}{N(p-1)}}-|\Omega|^{\frac{p-N}{N(p-1)}}\right)^{1/p^\prime}$}&\quad\text{if $|\Omega|<\infty,$}\\ \text{$ r^\frac{p-N}{Np}$}&\quad\text{if $|\Omega|=\infty$}.  
\end{cases}
\]
Now,  
$N>p$ 
implies
$\left(r^{\frac{p-N}{N(p-1)}}-|\Omega|^{\frac{p-N}{N(p-1)}}\right)^{1/p^\prime}\asymp r^{\frac{p-N}{Np}}$ near zero. 
Therefore, for all $r\in (0,|\Omega|)$ we have 
\[
\eta_{\Phi,\Psi}(r)\asymp
\begin{cases}
\text{$ r^\frac{N(p-q)+pq}{Np}$}&\quad\text{ near zero,}\\ \text{$ r^\frac{p}{N}$}&\quad\text{ near infinity when $|\Omega|=\infty$}.  
\end{cases}
\]
Consequently,  $\lim_{r\to 0}\eta_{\Phi,\Psi}(r)<\infty$ (as $q\leq p^*$), i.e. \eqref{7} holds.  Thus, the pair $(\Phi, \Psi)$ satisfies all the assumptions of Theorem \ref{theorem3} and hence  $X_{\Phi,\Psi}(\Omega)\subset \mathcal{H}_{\Phi,\Psi}(\Omega)$.  It is worth mentioning that $X_{\Phi,\Psi}(\Omega)=L^{\left(p^*/q\right)^\prime,\infty}(\Omega)$ when $|\Omega|<\infty.$ In general, $L^{\left(p^*/q\right)^\prime,\infty}(\Omega)\cap L^{N/p,\infty}(\Omega)\subset X_{\Phi,\Psi}(\Omega)$. 
 \end{enumerate}
\end{rem}

  \begin{center}
	{\bf Acknowledgments}
\end{center} 
Ujjal Das acknowledges the support of the Israel Science Foundation \!(grant $637/19$) founded by the
Israel Academy of Sciences and Humanities. Ujjal Das is also supported in
part by a fellowship from the Lady Davis Foundation.

\bibliographystyle{abbrvurl}
\bibliography{Reference}

\end{document}